\documentclass{amsart}
\usepackage{graphicx,verbatim, amsmath, amssymb, amsthm, amsfonts, epsfig, ifthen,mathtools,mathabx,epstopdf,enumerate,hyperref}	
\newtheorem{proposition}{Proposition}[section]
\newtheorem{theorem}[proposition]{Theorem}

\newtheorem{lemma}[proposition]{Lemma}
\newtheorem{prop}[proposition]{Proposition}
\newtheorem{cor}[proposition]{Corollary}
\newtheorem{conj}[proposition]{Conjecture}

\theoremstyle{definition}
\newtheorem{example}[proposition]{Example}
\newtheorem{definition}[proposition]{Definition}

\theoremstyle{remark}
\newtheorem{remark}[proposition]{Remark}

\numberwithin{equation}{section}
\usepackage[usenames]{color}

\usepackage[usenames]{color}

\newcounter{margincounter}
\newcommand{\newword}[1]{\textbf{\emph{#1}}}

\newcommand{\integers}{\mathbb Z}
\newcommand{\rationals}{\mathbb Q}

\newcommand{\reals}{\mathbb R}

\newcommand{\ep}{\varepsilon}

\newcommand{\sgn}{\operatorname{sgn}}

\newcommand{\set}[1]{{\left\lbrace #1 \right\rbrace}}

\newcommand{\A}{{\mathcal A}}

\newcommand{\F}{{\mathcal F}}

\newcommand{\dashname}[1]{\stackrel{#1}{\mbox{---\!---}}}

\renewcommand{\th}{^\mathrm{th}}

\newcommand{\g}{\mathbf{g}}
\renewcommand{\b}{\mathbf{b}}
\renewcommand{\k}{\mathbf{k}}
\renewcommand{\a}{\mathbf{a}}

\newcommand{\x}{\mathbf{x}}

\renewcommand{\v}{\mathbf{v}}

\newcommand{\tB}{\tilde{B}}
\newcommand{\T}{\mathbb{T}}

\newcommand{\M}{\mathcal{M}}

\renewcommand{\P}{\mathbb{P}}
\newcommand{\K}{\mathbb{K}}
\newcommand{\Trop}{\operatorname{Trop}}

\renewcommand{\S}{\mathbf{S}}
\renewcommand{\M}{\mathbf{M}}

\title{Universal geometric cluster algebras from surfaces}
\author{Nathan Reading}
\thanks{This material is based upon work partially supported by the Simons Foundation under Grant Number 209288 and by the National Science Foundation under Grant Number DMS-1101568.}

\begin{document}
\begin{abstract}
A universal geometric cluster algebra over an exchange matrix $B$ is a universal object in the category of geometric cluster algebras over $B$ related by coefficient specializations.
(Following an earlier paper on universal geometric cluster algebras, we broaden the definition of geometric cluster algebras relative to the definition originally given by Fomin and Zelevinsky.)
The universal objects are closely related to a fan $\F_B$ called the mutation fan for~$B$.
In this paper, we consider universal geometric cluster algebras and mutation fans for cluster algebras arising from marked surfaces.
We identify two crucial properties of marked surfaces:  The Curve Separation Property and the Null Tangle Property.
The latter property implies the former.
We prove the Curve Separation Property for all marked surfaces except once-punctured surfaces without boundary components, and as a result we obtain a construction of the rational part of $\F_B$ for these surfaces.
We prove the Null Tangle Property for a smaller family of surfaces and use it to construct universal geometric coefficients for these surfaces.
\end{abstract}
\maketitle

\setcounter{tocdepth}{1}
\tableofcontents

\section{Introduction}
The combinatorial datum that specifies a cluster algebra of geometric type is an extended exchange matrix $\tB$, which consists of a skew-symmetrizable integer matrix $B$, called an exchange matrix, and some additional rows called the coefficient rows of~$\tB$.
Certain exchange matrices $B$ are associated to tagged triangulations of marked surfaces, as explained in \cite{cats1,cats2}. 
In this case, the choice of coefficient rows amounts to choosing a set of laminations in the surface.
A lamination is a collection of weighted curves in the surface.
The corresponding coefficient row is a vector known as the shear coordinates of the lamination.
The notion of laminations predates the notion of cluster algebras.
References are found in \cite[Sections~12--13]{cats2}.

Given an exchange matrix $B$, we consider the category of cluster algebras of geometric type associated to $B$, with maps given by certain specializations of the coefficients.
We follow \cite{universal} in broadening the category of cluster algebras of geometric type to allow infinitely many coefficient rows with not-necessarily integral entries, and then narrowing the definition of coefficient specialization, as compared with the definitions in~\cite{ca4}. 
The category also depends on an underlying ring $R$, usually $\integers$, $\rationals$, or $\reals$.
If this category contains a universal object, the universal object is called the universal geometric cluster algebra for $B$ over~$R$.
Universal geometric cluster algebras over $R$ are known to exist \cite[Corollary~4.7]{universal} when $R$ is a field, but there is no general theorem guaranteeing the existence of universal geometric cluster algebras over $\integers$.
However, the existence of universal geometric cluster algebras over $\integers$ is proved for various $B$ in \cite{universal,unitorus}, and for some additional $B$ in this paper, as we explain below.

Constructing a universal geometric cluster algebra is equivalent to constructing a universal extended exchange matrix over $B$ by specifying a collection of coefficient rows, or universal geometric coefficients over $R$.
The main theme of \cite{universal} is to make the connection between universal geometric coefficients and the piecewise-linear geometry of a fan $\F_B$, called the mutation fan.

The goal of this paper is to construct the mutation fan and universal geometric coefficients in the case of marked surfaces.
Both constructions involve a variant on laminations that we call quasi-laminations.
We show that the fan defined naturally by quasi-laminations equals the rational part of the mutation fan if and only if the surface has a property that we call the Curve Separation Property (Theorem~\ref{rat FB surfaces}).
In most cases, the Curve Separation Property is the following:
Given two incompatible allowable curves (in the sense of quasi-laminations), there exists a tagged triangulation of the surface such that the shear coordinates of the two curves have strictly opposite signs in at least one entry.
However, in the case of a once-punctured surface without boundary, the statement of the Curve Separation Property must be altered slightly.
(See Definition~\ref{curve sep def}.)
A long but elementary argument establishes the Curve Separation Property for all marked surfaces except once-punctured surfaces without boundary (Theorem~\ref{curve separation}).

We complete the construction (Corollary~\ref{poly grow univ}) of universal geometric coefficients (over $\integers$ or $\rationals$) for a smaller family of surfaces.
Specifically, we show in Theorem~\ref{null tangle theorem} that the shear coordinates of allowable curves (in the sense of quasi-laminations) are universal geometric coefficients if and only if the surface has a stronger property called the Null Tangle Property:
If a finite collection of weighted curves has the weighted sum of its shear coordinates equal to zero for all tagged triangulations, then all weights are zero.
(Once again, the statement of the property must be modified for once-punctured surfaces without boundary.)
We prove the Null Tangle Property in certain cases by extending some of the arguments that establish the Curve Separation Property.
Specifically, Theorem~\ref{poly grow null tangle} establishes the Null Tangle Property for a disk with at most $2$ punctures, for an annulus with at most~$1$ puncture, and for a sphere with 3 boundary components and no punctures.

It remains open to determine whether once-punctured surfaces without boundary have the Curve Separation Property, and to determine which additional surfaces have the Null Tangle Property.
However, in \cite{unitorus}, we establish the Null Tangle Property (and thus the Curve Separation Property) in the simplest case not covered in this paper:  the once-punctured torus.
This leads to an explicit construction of universal geometric coefficients over $\integers$, $\rationals$ or $\reals$ for the once-punctured torus.

\section{Cluster algebras with universal geometric coefficients}
In this section, we briefly review the construction of cluster algebras of geometric type, in the broader sense introduced in~\cite{universal}.
We describe coefficient specialization and characterize the universal objects under specialization, the universal geometric cluster algebras or cluster algebras with universal geometric coefficients.
We also present some results from~\cite{universal} that are useful in constructing universal geometric coefficients for cluster algebras from surfaces.
All of the constructions depend on a choice of an \newword{underlying ring} $R$.
We require that $R$ is either the integers $\integers$ or a field containing the rationals $\rationals$ as a subfield and contained as a subfield of the real numbers $\reals$.
We are most interested in the cases where $R$ is $\integers$, $\rationals$, or $\reals$.

%\subsection{Cluster algebras of geometric type}\label{geom type subsec}
Let $I$ be an indexing set of arbitrary cardinality.
The usual definition of cluster algebras of geometric type, found in \cite[Section~2]{ca4}, is recovered from the definition below by taking $R=\integers$ and $I$ finite.

For each $i\in I$, let $u_i$ be a formal symbol.
Let $\Trop_R(u_i:i\in I)$ be the set of formal products of the form $\prod_{i\in I}u_i^{a_i}$ with each $a_i$ in $R$.
This is an abelian group with product
\[\prod_{i\in I}u_i^{a_i}\cdot\prod_{i\in I}u_i^{b_i}=\prod_{i\in I}u_i^{a_i+b_i}\]
and identity $\prod_{i\in I}u_i^0$.
We define an \newword{auxiliary addition} $\oplus$ in $\Trop_R(u_i:i\in I)$ by 
\[\prod_{i\in I}u_i^{a_i}\oplus\prod_{i\in I}u_i^{b_i}=\prod_{i\in I}u_i^{\min(a_i,b_i)}.\]
The triple $(\Trop_R(u_i:i\in I),\,\oplus,\,\cdot\,)$ is called a \newword{Tropical semifield} over $R$.
The symbols $u_i$ are called \newword{Tropical variables}.
We also consider $\Trop_R(u_i:i\in I)$ as a topological space.
It is the product, indexed by $I$, of copies of $R$.
We give $R$ the discrete topology and give $\Trop_R(u_i:i\in I)$ the product topology as a product of copies of the discrete set $R$.
Details about the product topology may be found, for example, in \cite[Chapter~3]{Kelley}.
See also \cite[Section~3]{universal}.  %\margin{This is section {univ sec}.}

An \newword{exchange matrix} $B$ is a \newword{skew-symmetrizable} $n\times n$ integer matrix. 
That means that there exist positive integers $d_1,\ldots,d_n$ such that $d_ib_{ij}=-d_jb_{ji}$ for all $i,j\in[n]$.  
In this paper, we will work with \newword{skew-symmetric} exchange matrices, meaning that $b_{ij}=-b_{ji}$ for all $i,j\in[n]$.  
(The notation $[n]$ stands for the indexing set $\set{1,\ldots,n}$ of $B$.)
An \newword{extended exchange matrix} is $B$, augmented by an additional collection of \newword{coefficient rows}, indexed by $I$.
Each coefficient row is a vector in $R^n$.
If $I$ is infinite, then $\tB$ is not a matrix in the usual sense, but it is a collection of entries $b_{ij}$ where $i$ is either in $[n]$ or in $I$ and $j$ is in $[n]$.

Now set $\P=\Trop_R(u_i:i\in I)$ and let $\K$ be a field of rational functions in $n$ independent variables with coefficients in $\rationals\P$.
Here $\rationals\P$ is the group ring of the group $\P$, ignoring the auxiliary addition $\oplus$.
A \newword{(labeled geometric) seed} of rank $n$ is a pair $(\x,\tB)$ such that $\x=(x_1,\ldots,x_n)$ is an $n$-tuple of algebraically independent elements of $\K$ and $\tB$ is an extended exchange matrix with coefficient rows indexed by $I$.
The $n$-tuple $\x$ is called a \newword{cluster} and the entries of $\x$ are called \newword{cluster variables}.
The seed determines elements $y_j=\prod_{i\in I}u_i^{b_{ij}}$ of $\P$ called \newword{coefficients}.

Given a seed $(\x,\tB)$ and an index $k\in[n]$, we define an new seed $\mu_k(\x,\tB)=(\x',\tB')$.
The new cluster $\x'=(x_1',\ldots,x_n')$ has $x_j'=x_j$ whenever $j\neq k$ and 
\begin{equation}\label{x mut}
x_k'=x_k^{-1}\left(\prod_{i=1}^nx_i^{[b_{ik}]_+}\prod_{i\in I}u_i^{[b_{ik}]_+}+\prod_{i=1}^nx_i^{[-b_{ik}]_+}\prod_{i\in I}u_i^{[-b_{ik}]_+}\right),
\end{equation}
where $[a]_+$ stands for $\max(a,0)$.
The new cluster variable $x_k'$ is a rational function in $\x$ with coefficients in $\integers\P$, so that $x_k'\in\K$.

The new extended exchange matrix $\tB'$ is given by
\begin{equation}\label{b mut}
b_{ij}'=\left\lbrace\!\!\begin{array}{ll}
-b_{ij}&\mbox{if }i=k\mbox{ or }j=k;\\
b_{ij}+\sgn(b_{kj})\,[b_{ik}b_{kj}]_+&\mbox{otherwise,}
\end{array}\right.
\end{equation}
where $\sgn(a)$ means $a/|a|$ if $a\neq 0$ or $0$ if $a=0$.
Since $\tB'$ only depends on $\tB$ and $k$, not on $\x$, we also write $\mu_k(\tB)$ for $\tB'$.
For any finite sequence $\k=k_q,\ldots,k_1$ of integers in~$[n]$, we use $\mu_\k$ for $\mu_{k_q}\circ\mu_{k_{q-1}}\circ\cdots\circ\mu_{k_1}$.

Now let $\T_n$ be the $n$-regular tree with an edge-labeling by integers $1$ through $n$, with each vertex incident to exactly one edge with each label.
An edge labeled $k$ between vertices $t$ and $t'$ is denoted by $t\dashname{k}t'$.
We fix a vertex $t_0$ and define a map $t\mapsto(\x_t,\tB_t)$ from vertices of $\T_n$ to seeds by requiring that $t_0\mapsto (\x,\tB)$ and that $(\x_{t'},\tB_{t'})=\mu_k(\x_t,\tB_t)$ whenever $t\dashname{k}t'$.
The map $t\mapsto(\x_t,\tB_t)$ is a \newword{cluster pattern} and the map $t\mapsto\tB_t$ is a \newword{$Y$-pattern}.
We write $(x_{1;t},\ldots,x_{n;t})$ and $(b_{ij}^t)$ for the cluster variables and matrix entries of $(\x_t,\tB_t)$ and write $y_{j;t}$ for the coefficient $\prod_{i\in I}u_i^{b_{ij}^t}$.
Associated to the cluster pattern is a \newword{cluster algebra $\A_R(\x,\tB)$ of geometric type}, namely the $\integers\P$-subalgebra of $\K$ generated by all cluster variables in the cluster pattern.
The cluster algebra is determined up to isomorphism by the extended exchange matrix $\tB$, so we write $\A_R(\tB)$ for $\A(\x,\tB)$.

For any permutation $\pi$ of $[n]$, applying $\pi$ simultaneously to the entries of $\x$, to the rows of $B$, to the columns of $\tB$ and to the edge-labels of $\T_n$ induces a canonical isomorphism of cluster patterns and of cluster algebras.  
In \cite[Definition~4.1]{ca4} \newword{unlabeled seeds} are defined as equivalence classes of seeds under these simultaneous permutations.
The \newword{exchange graph} is the graph on unlabeled seeds obtained from the cluster pattern by identifying vertices that have the same \emph{unlabeled seed}.

Let $\P=\Trop_R(u_i:i\in I)$ and $\P'=\Trop_R(v_k:k\in K)$ be tropical semifields over the same underlying ring $R$.
Let $(\x,\tB)$ and $(\x',\tB')$ be seeds of rank $n$, with coefficient rows indexed by $I$ and $K$ respectively.
A \newword{coefficient specialization} is a ring homomorphism $\varphi:\A_R(\x,\tB)\to\A_R(\x',\tB')$ such that
 \begin{enumerate}[(i)]
\item the exchange matrices $B$ and $B'$ coincide;
\item $\varphi(x_j)=x'_j$ for all $j\in[n]$;
\item the restriction of $\varphi$ to $\P$ is a continuous $R$-linear map to $\P'$ with $\varphi(y_{j;t})=y'_{j;t}$ and $\varphi(y_{j;t}\oplus 1)=y'_{j;t}\oplus 1$ for all $j\in[n]$ and $t\in\T_n$.
Continuity here refers to the product topology described above.
\end{enumerate}
The cluster algebra $\A=\A_R(\tB)$ has \newword{universal geometric coefficients over $R$} if the following condition holds:
Given a cluster algebra $\A'=\A_R(\tB')$ with the same initial exchange matrix $B$, there exists a unique coefficient specialization from $\A$ to $\A'$.
The extended exchange matrix $\tB$ is then said to be \newword{universal over $R$}.

Equation~\eqref{b mut} leads to the definition of \newword{mutation maps} for $B$.  
For each $k\in[n]$, we define a map $\eta^B_k:\reals^n\to\reals^n$ by setting $\eta_k^B(a_1,\ldots,a_n)=(a_1',\ldots,a_n')$ with
\begin{equation}\label{mutation map def}
a'_j=\left\lbrace\begin{array}{ll}
-a_k&\mbox{if }j=k;\\
a_j+a_kb_{kj}&\mbox{if $j\neq k$, $a_k\ge 0$ and $b_{kj}\ge 0$};\\
a_j-a_kb_{kj}&\mbox{if $j\neq k$, $a_k\le 0$ and $b_{kj}\le 0$};\\
a_j&\mbox{otherwise.}
\end{array}\right.
\end{equation}
Equivalently, if $\tB$ is an extended exchange matrix with exchange matrix $B$ and $(a_1,\ldots,a_n)$ is a coefficient row of $\tB$, then $\eta_k^B(\a)$ is the corresponding coefficient row of $\mu_k(\tB)$.
More generally, the  mutation maps associated to $B$ are defined as follows.
Let $\k=k_q,k_{q-1},\ldots,k_1$ be a sequence of integers in $[n]$, let $B_1=B$ and define $B_{i+1}=\mu_{k_i,\ldots,k_1}(B)$ for $i\in[q]$.
Then
\begin{equation}\label{eta def}
\eta_\k^B=\eta^B_{k_q,k_{q-1},\ldots,k_1}=\eta_{k_q}^{B_{q}}\circ\eta_{k_{q-1}}^{B_{q-1}}\circ\cdots\circ\eta_{k_1}^{B_1}.
\end{equation}
The map $\eta_\k^B$ is continuous, piecewise linear, and invertible, with inverse $\eta^{B_{q+1}}_{k_1,\ldots,k_q}$.
In particular, it is a homeomorphism from $\reals^n$ to $\reals^n$.
It also has the property that 
\begin{equation}\label{eta antipodal}
\eta_\k^B(\a)=-\eta_\k^{-B}(-\a).
\end{equation}

Given a finite set $S$, vectors $(\v_i:i\in S)$ in $\reals^n$, and elements $(c_i:i\in S)$ of $R$, write the formal expression $\sum_{i\in S}c_i\v_i$.
This expression is a \newword{$B$-coherent linear relation with coefficients in $R$} if the equalities
\begin{eqnarray}
\label{linear eta}
&&\sum_{i\in S}c_i\eta^B_\k(\v_i)=\mathbf{0},\mbox{ and}\\
\label{piecewise eta}
&&\sum_{i\in S}c_i\mathbf{min}(\eta^B_\k(\v_i),\mathbf{0})=\mathbf{0}
\end{eqnarray}
hold for every finite sequence $\k=k_q,\ldots,k_1$ of integers in~$[n]$.
The symbol $\mathbf{0}$ is the zero vector and $\mathbf{min}$ refers to componentwise minimum.
The equality \eqref{linear eta} for the empty sequence $\k$ implies that a $B$-coherent linear relation is also a linear relation in the usual sense. 

Let $I$ be some indexing set and let $(\b_i:i\in I)$ be vectors in $R^n$.
Then $(\b_i:i\in I)$ is an \newword{$R$-spanning set for $B$} if the following condition holds:
If $\a\in R^n$, then there exists a finite subset $S\subseteq I$ and elements $(c_i:i\in S)$ of $R$ such that $\a-\sum_{i\in S}c_i\b_i$ is a $B$-coherent linear relation.
Also, $(\b_i:i\in I)$ is an \newword{$R$-independent set for $B$} if the following condition holds: 
If $S$ is a finite subset of $I$ and $\sum_{i\in S}c_i\b_i$ is a $B$-coherent linear relation with coefficients in $R$, then $c_i=0$ for all $i\in S$.
Finally, $(\b_i:i\in I)$ is an \newword{$R$-basis} for $B$ if it is both an $R$-independent set for $B$ and an $R$-spanning set for $B$.

The notion of an $R$-basis for $B$ is analogous to the linear-algebraic notion of a basis, with $B$-coherent linear relations replacing the usual linear relations.
Our interest in $R$-bases for $B$ arises from the following theorem, which is \cite[Theorem~4.4]{universal}. %\margin{This is Thm. {basis univ}}

\begin{theorem}\label{basis univ}
Let $\tB$ be an extended exchange matrix with entries in $R$.
Then $\tB$ is universal over $R$ if and only if the coefficient rows of $\tB$ are an $R$-basis for $B$.
\end{theorem}
For a description of the unique coefficient specialization from $\A_R(\tB)$ to any other cluster algebra with the same $R$ and $B$, see \cite[Remark~4.5]{universal}. %\margin{This is Remark {explicit spec}} 

The following proposition \cite[Proposition~4.6]{universal} %\margin{This is prop. {basis exists}} 
is proved by the usual argument, using Zorn's Lemma, that any vector space has a (Hamel) basis.

%sinceTAMS:  Fixed following proposition (thanks, ref!) and comment immediately after.
\begin{prop}\label{basis exists}
Suppose the underlying ring $R$ is a field.
For any exchange matrix $B$, there exists an $R$-basis for $B$.
Given an $R$-spanning set $U$ for $B$, there exists an $R$-basis for $B$ contained in $U$. 
\end{prop}
In light of Theorem~\ref{basis univ} and Proposition~\ref{basis exists}, if $R$ is a field, then universal geometric coefficients over $R$ exist for any $B$.
We have no general proof that $\integers$-bases exist, and indeed the second statement of Proposition~\ref{basis exists} can fail for $R=\integers$.
(See \cite[Remark~4.9]{universal}.)
However, $\integers$-bases (and thus universal geometric coefficients over $\integers$) are constructed for various exchange matrices $B$ in \cite{universal,unitorus} and also later in this paper (Corollary~\ref{poly grow univ}).

%sinceTAMS:  rephrased following sentence and then removed 
%Proposition {one positive or one negative}
The following proposition is \cite[Proposition~4.11]{universal}. %\margin{This is Prop {no zero no piecewise}} 
\begin{prop}\label{no zero no piecewise}
Suppose $B$ has no row consisting entirely of zeros.
Then the formal expression $\sum_{i\in S}c_i\v_i$ is a $B$-coherent linear relation if and only if condition~\eqref{linear eta} holds for all sequences $\k$.
\end{prop}

A subset of $\reals^n$ that is closed under addition and positive scaling is a \newword{convex cone}.
A \newword{simplicial} cone is the nonnegative span of a finite set of linearly independent vectors.
A \newword{rational cone} is the nonnegative span of a finite set of rational vectors.
A \newword{face} of a convex cone $C$ is a convex subset $F$ such that any line segment $L\subseteq C$ whose interior intersects $F$ has $L\subseteq F$.
Any face of a convex cone is itself a convex cone.
A face of a closed convex cone is closed as well.
A collection $\F$ of closed convex cones is a \newword{fan} if it has the following two properties:
If $C\in\F$ and $F$ is a face of $C$, then $F\in\F$; and the intersection of any two cones in $\F$ is a face of each of the two.
A \newword{complete fan} is a fan $\F$ such that the union of the cones in $\F$ is the entire ambient space.
A fan is \newword{simplicial} if all of its cones are simplicial.
A fan is \newword{rational} if all of its cones are rational.
If $\F_1$ and $\F_2$ are fans and every cone in $\F_2$ is a union of cones in $\F_1$, then $\F_1$ \newword{refines}~$\F_2$.

As before, $\sgn(a)$ means $a/|a|$ if $a\neq 0$ or $0$ if $a=0$.
Given $\a=(a_1,\ldots,a_n)$ in $\reals^n$, we write $\mathbf{sgn}(\a)$ for $(\sgn(a_1),\ldots,\sgn(a_n))$.
Given $\a_1,\a_2\in\reals^n$, define $\a_1\equiv^B\a_2$ to mean that $\mathbf{sgn}(\eta^B_\k(\a_1))=\mathbf{sgn}(\eta^B_\k(\a_2))$ for every sequence $\k$ of integers in~$[n]$.
A $\equiv^B$-equivalence class is a \newword{$B$-class}.
The closure of a $B$-class is a \newword{$B$-cone}.
Each $B$-class is a convex cone and each $B$-cone is a closed convex cone \cite[Proposition~5.4]{universal}.  %\margin{This is prop. {convex}.}
The \newword{mutation fan} for $B$ is the collection $\F_B$ consisting of all $B$-cones and all faces of $B$-cones.
(Faces of $B$-cones can fail to be $B$-cones.  See \cite[Remark~5.28]{universal}.)
The following results are \cite[Proposition~5.3]{universal}, %\margin{This is prop. {linear}.} 
\cite[Theorem~5.13]{universal}, and  %\margin{This is Theorem {fan}.}
\cite[Proposition~5.26]{universal}. % \margin{This is prop. {int of B-cones}.} 

\begin{prop}\label{linear}
Every mutation map $\eta_\k^B$ is linear on every $B$-cone.
\end{prop}

\begin{theorem}\label{fan}
The collection $\F_B$ is a complete fan.
\end{theorem}

\begin{prop}\label{int of B-cones}
An arbitrary intersection of $B$-cones is a $B$-cone.
\end{prop}

A collection of vectors in $\reals^n$ is \newword{sign-coherent} if their $k\th$ coordinates weakly agree in sign, for each $k\in[n]$.
The following proposition is \cite[Proposition~5.30]{universal}. %\margin{This is prop. {contained Bcone}.} 

\begin{prop}\label{contained Bcone}
A set $C\subseteq\reals^n$ is contained in some $B$-cone if and only if the set $\eta_\k^B(C)$ is sign-coherent for every sequence $\k$ of indices in $[n]$.
\end{prop}

An $R$-basis $(\b_i:i\in I)$ for $B$ is \newword{positive} if, for any vector $\a\in R^n$, there is a $B$-coherent linear relation $\a-\sum_{i\in S}c_i\b_i$ with each $c_i$ nonnegative.
A positive $R$-basis seems to be the natural choice for ``the right'' $R$-basis for $B$, and this intuition is borne out by the following results, which are \cite[Proposition~6.2]{universal}  %\margin{This is Prop. {positive unique}.} 
and \cite[Corollary~6.13]{universal}. %\margin{This is Cor. {pos reals FB}.}
Note the appearance of $\reals$ in Proposition~\ref{pos reals FB}, rather than $R$.

\begin{prop}\label{positive unique}
For any fixed $R$, there is at most one positive $R$-basis for $B$, up to scaling of each basis element by a positive unit in $R$.
\end{prop}

\begin{prop}\label{pos reals FB}
If a positive $\reals$-basis for $B$ exists, then $\F_B$ is simplicial.  
The basis consists of exactly one vector in each ray of $\F_B$.
\end{prop}

The following proposition is part of \cite[Proposition~6.7]{universal}.  %\margin{This is prop. {positive cone basis}.}
\begin{prop}\label{positive cone basis}
Suppose $(\b_i:i\in I)$ is an $R$-independent set for $B$ with the following property:
If $C$ is a $B$-cone, then the nonnegative $R$-linear span of $\set{\b_i:i\in I}\cap C$ contains $R^n\cap C$.
Then $(\b_i:i\in I)$ is a positive $R$-basis for $B$.
\end{prop}

\section{Cluster algebras from triangulated surfaces}\label{surface sec}
In this section, we give background material on cluster algebras arising from triangulated surfaces, following \cite{cats1,cats2}.

\begin{definition}[\emph{Bordered surface with marked points}]\label{S M def}
Let $\S$ be a surface obtained from a compact, oriented surface without boundary by deleting a finite collection of disjoint open disks.
The boundary of $\S$ is a finite collection of circles called \newword{boundary components}, one for each disk removed.
Let $\M$ be a nonempty, finite collection of points in $\S$ called \newword{marked points}, with at least one marked point in each boundary component.
Marked points in the interior of $\S$ are called \newword{punctures}.
A curve contained in the boundary of $\S$, connecting two points in $\M$, with no points of $\M$ in the interior of the curve, is a \newword{boundary segment}.  
We exclude unpunctured monogons, digons, and triangles, as well as once-punctured monogons.
If $\S$ is a sphere, then we require that it have at least $4$ punctures.  
In \cite{cats1,cats2}, the surface $\S$ is considered to be a Riemann surface, but the complex structure is not needed in this paper.
\end{definition}

\begin{definition}[\emph{Arcs and triangulations}]\label{tri def}
%sinceTAMS:  minor rephasings here.
An \newword{arc} in $(\S,\M)$ is a curve in $\S$, up to isotopy relative to $\M$, whose endpoints are points in $\M$.
We require that the arc not intersect itself, except that its endpoints may coincide, and that the arc is disjoint from $\M$ and from the boundary of $\S$, except at its endpoints.
We also exclude arcs that bound an unpunctured monogon and arcs that, together with a boundary segment connecting the endpoints, define an unpunctured digon.
Two arcs $\alpha$ and $\gamma$ are \newword{incompatible} if they intersect in $\S\setminus\M$ and if the intersection cannot be removed by (independently) isotopically deforming the arcs.
Otherwise, they are \newword{compatible}.
%sinceTAMS:  got rid of the adjective "ideal" throughout and added an explanation.
A \newword{triangulation} is a maximal collection of distinct pairwise compatible arcs.
The exclusion of certain pairs $(\S,\M)$ in Definition~\ref{S M def} assures that $(\S,\M)$ admits at least two triangulations.
The arcs defining a triangulation divide $\S$ into \newword{triangles}, which have $1$, $2$, or $3$ distinct vertices and $2$ or $3$ distinct sides.
(In \cite{cats1,cats2}, the terms \newword{ideal triangles} and \newword{ideal triangulations} are used because the marked points lie at infinity in the complex structure.)
A \newword{self-folded triangle} is a triangle with 2 distinct sides.
(For example, later in Figure~\ref{shear self-folded fig}, one of the two triangles shown in each picture is self-folded.)
An arc in a triangulation is an edge of exactly two triangles in the triangulation unless it constitutes two sides of a self-folded triangle.
Removing the arc from the triangulation leaves a quadrilateral region, which can be completed to two triangles by either one of the diagonals of the quadrilateral.
A \newword{flip} is the operation of removing an arc from a triangulation and forming a new triangulation by inserting the other diagonal of the quadrilateral.
\end{definition}

\begin{definition}[\emph{Tagged arcs and tagged triangulations}]\label{tagged tri def}
A \newword{tagged arc} in $(\S,\M)$ is an arc in $\S$ that does not cut out a once-punctured monogon, with each end of the arc designated (or ``tagged'') either \newword{plain} or \newword{notched}.
An endpoint on a boundary component must be tagged plain, and if the endpoints of an arc coincide, then they must have the same tag.
Two tagged arcs are \newword{compatible} if one of the following conditions holds:
the two underlying untagged arcs are distinct (i.e.\ nonisotopic) and compatible and any shared endpoints have the same tagging, or
the two underlying untagged arcs are the same and their tagging agrees on one endpoint but not the other.
A \newword{tagged triangulation} is a maximal collection of distinct pairwise compatible tagged arcs.

There is a map $\tau$ from ordinary arcs to tagged arcs, defined as follows:
If an arc $\gamma$ does not bound a once-punctured monogon, then $\tau(\gamma)$ is the same curve, with both endpoints tagged plain.
If $\gamma$ has both endpoints at $a\in\M$  and bounds a once-punctured monogon with puncture $b\in\M$, then $\tau(\gamma)$ is the arc, in the monogon, from $a$ to $b$, tagged plain at $a$ and notched at $b$.
Starting with a triangulation, and applying $\tau$ to each arc, we obtain a tagged triangulation.

As an example of the map $\tau$, consider the triangulation $T$ of a punctured digon shown in the left picture of Figure~\ref{tau fig}.
One arc defining $T$ connects a marked point on the boundary to the puncture, and the other arc (in purple or light gray) encloses the puncture.
The right picture of the figure shows the tagged triangulation $\tau(T)$.
\begin{figure}[ht]
\scalebox{1.2}{\includegraphics{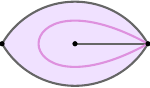}}\qquad\qquad\scalebox{1.2}{\includegraphics{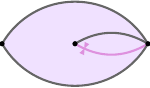}}
\caption{The map $\tau$ from ordinary arcs to tagged arcs}
\label{tau fig}
\end{figure}
\end{definition}

\begin{definition}[\emph{Tagged arc complex and tagged flips}]\label{complex and flips}
The tagged triangulations are the maximal simplices of a simplicial complex called the \newword{tagged arc complex}.
This is the complex whose vertices are the tagged arcs, with a set of arcs defining a face if and only if they are pairwise compatible.
All tagged triangulations of $(\S,\M)$ have the same number of tagged arcs.
(See \cite[Proposition~2.10]{cats1}.)
Given a tagged triangulation $T$ and a particular tagged arc $\gamma$ in $T$, there is a unique tagged triangulation $T'$ such that $T'\neq T$ and $T\setminus\set{\gamma}\subset T'$.
The tagged triangulations $T$ and $T'$ are said to be related by a \newword{tagged flip}, or more specifically, $T'$ is obtained from $T$ by a flip at $\gamma$.
An exact combinatorial description of tagged flips is found in \cite[Remark~5.13]{cats2}.
The dual graph to the tagged arc complex is the graph whose vertices are tagged triangulations and whose edges are pairs of triangulations related by a tagged flip.
This dual graph is connected \cite[Proposition~7.10]{cats1}, except when $\S$ has no boundary components and exactly one puncture.
In the latter case, the dual graph of the tagged arc complex has exactly two connected components, one consisting of triangulations with all arcs tagged plain and the other consisting of triangulations with all arcs tagged notched.
The two components are related by the symmetry of reversing all taggings (plain$\leftrightarrow$notched).
\end{definition}

Given a tagged triangulation $T$, there is a simple recipe \cite[Definition~5.15]{cats2} for writing an exchange matrix $B(T)$, called the \newword{signed adjacency matrix} of~$T$.
This is an $n\times n$ matrix whose rows and columns are indexed by the tagged arcs $\gamma$ in~$T$.
We do not need the details of this definition.
Rather, what is important for us are some observations and results of \cite{cats1,cats2}.
First, $B(T)$ is skew-symmetric, rather than merely skew-symmetrizable.
Next, $B(T)$ does not have a row or column consisting of all zeros.
Finally, we quote very small parts of \cite[Theorem~6.1]{cats2} and \cite[Corollary~6.2]{cats2}.

\begin{prop}\label{signed adjacency mut}
If $T$ and $T'$ are tagged triangulations such that $T'$ is obtained from $T$ by a flip at an arc $\gamma$, then $B(T')$ is obtained from $B(T)$ by matrix mutation at the position indexed by $\gamma$. 
\end{prop}

\begin{theorem}\label{cluster main}
Let $\A$ be a cluster algebra with initial exchange matrix $B(T)$ and arbitrary coefficients.
Then the exchange graph of $\A$ is isomorphic to the dual graph of the tagged arc complex or, if $\S$ has no boundary components and exactly one puncture, the connected component of the dual graph consisting of tagged triangulations with all arcs tagged plain.
\end{theorem}

\begin{definition}[\emph{Integral and rational laminations}]\label{lam def}
An \newword{integral (unbounded measured) lamination} is a collection of pairwise nonintersecting curves in $\S$, up to isotopy relative to $\M$.
We take the curves to be distinct up to isotopy, but each curve $\lambda$ appears with a positive integer weight $w_\lambda$.
Each curve must be
\begin{itemize}
\item a closed curve,
\item a curve whose two endpoints are unmarked boundary points,
\item a curve having one endpoint an unmarked boundary point, with the other end spiraling (clockwise or counterclockwise) into a puncture, or
\item a curve spiraling into (not necessarily distinct) punctures at both ends.
\end{itemize}
However, we exclude
\begin{itemize}
\item curves with self-intersections,
\item curves that are contractible in $\S\setminus\M$,
\item curves that are contractible to a puncture, and
\item curves having two endpoints on the boundary and being contractible to a portion of the boundary containing zero or one marked points.
\end{itemize}
When a curve spirals into a puncture, that puncture is a \newword{spiral point} of the curve.

A \newword{rational (unbounded measured) lamination} is defined in the same way, but with positive rational weights rather than positive integer weights.
\end{definition}

\begin{definition}[\emph{Shear coordinates}]\label{shear def}
Given a lamination $L$ and a tagged triangulation $T$,  the \newword{shear coordinate vector} is a vector $\b(T,L)=(b_\gamma(T,L):\gamma\in T)$ indexed by arcs $\gamma$ in $T$.
The entry $b_\gamma(T,L)$ is defined as follows.
We deal with each curve $\lambda\in L$ separately and set
\begin{equation}\label{each curve separately}
b_\gamma(T,L)=\sum_{\lambda\in L}w_\lambda b_\gamma(T,\lambda),
\end{equation}
where $w_\lambda$ is the weight of $\lambda$ in $L$ and where $b_\gamma(T,\lambda)$ is a quantity that we now define.
First, everywhere there is a puncture $p$ incident only to notched ends of arcs, change all of these notched ends to plain ends and reverse the direction of any spirals of $\lambda$ at $p$.
We obtain an altered tagged triangulation with no notched endpoints on arcs, except that at some points $b$ there are two tagged arcs, whose underlying untagged arcs are the same, one tagged plain and one tagged notched at~$b$.
If $\gamma$ is one of these arcs for some $b$, then if necessary we switch notched ends with plain ends and reverse any spirals of $\lambda$ at $b$ so that $\gamma$ becomes an arc with one end notched.
Call the altered tagged triangulation $T'$ and call the curve with some spirals possibly reversed $\lambda'$.
Define $b_\gamma(T,\lambda)$ to be $b_\gamma(T',\lambda')$, calculated as follows.

There is a triangulation $T^\circ$ with $\tau(T^\circ)=T'$.
Reuse the name $\gamma$ for the arc in $T^\circ$ corresponding to $\gamma\in T$.
Because of how we defined $T'$, the arc $\gamma$ is contained in two distinct triangles of $T^\circ$ (rather than defining two edges of the same self-folded triangle).
Choose an isotopy deformation of $\lambda'$ (reusing the symbol $\lambda'$ for this deformation) so as to minimize the number of intersection points of $\lambda'$ with $\gamma$.
Then $b_\gamma(T',\lambda')$ is the sum, over each intersection of $\lambda'$ with $\gamma$, of a number in $\set{-1,0,1}$.
The number is $0$ unless $\lambda'$ intersects the two triangles containing $\gamma$ as shown in Figure~\ref{shear fig}, in which case it is $1$ or $-1$ as shown.
\begin{figure}[ht]
\raisebox{42 pt}{$+1$}\,\,\includegraphics{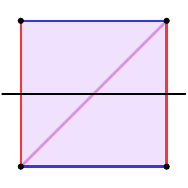}\qquad\qquad\includegraphics{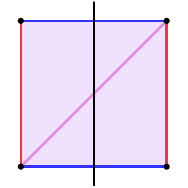}\,\,\raisebox{42 pt}{$-1$}
\caption{Computing shear coordinates:  basic case}
\label{shear fig}
\end{figure}
In both pictures, the arc $\gamma$ is the diagonal of the square and the curve $\lambda'$ is the vertical or horizonal line intersecting the square.
Figure~\ref{shear fig} introduces a color convention that we maintain throughout the paper.
When calculating a shear coordinate at $\gamma$, we color $\gamma$ purple and color the other edges red and blue so that, traveling clockwise along each triangle, we meet the colors in the order red, blue, purple.  
Those reading the paper in grayscale may have some difficulty in distinguishing the arcs;  it may be helpful to know that the blue arcs become a darker gray than the red arcs and the purple arcs are lightest of all.
We also shade the interiors of the red-blue-purple triangles with a light purple.
The quadrilaterals represented in Figure~\ref{shear fig} might have fewer than $4$ distinct vertices and $4$ distinct sides.
Many examples appear later on.
The first example occurs in Figure~\ref{shear self-folded fig}, which illustrates shear coordinates in the case where one of the two triangles in Figure~\ref{shear fig} is self-folded.
\begin{figure}[ht]
\raisebox{33 pt}{$+1$}\,\,\scalebox{1.2}{\includegraphics{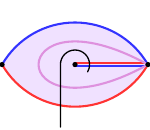}}\qquad\qquad\scalebox{1.2}{\includegraphics{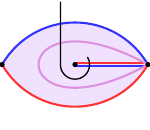}}\,\,\raisebox{33 pt}{$-1$}
\caption{Computing shear coordinates: self-folded case}
\label{shear self-folded fig}
\end{figure}
Although a red and blue edge coincide, we show both colors for clarity, putting the colors on the appropriate sides to make the connection to Figure~\ref{shear fig}.
\end{definition}

The following are (a part of) \cite[Theorem~13.5]{cats2} and \cite[Theorem~13.6]{cats2}.
\begin{theorem}\label{mut lam}
Fix a lamination $L$.
If $T$ and $T'$ are tagged triangulations such that $T'$ is obtained from $T$ by a flip at an arc $\gamma$, then $\b(T',L)=\eta_\gamma^{B(T)}(\b(T,L))$, where $\eta_\gamma^{B(T)}$ is the mutation map at the index for the arc $\gamma$ in~$T$.
\end{theorem}

\begin{theorem}\label{lam bij}
Fix a tagged triangulation $T$.
Then the map $L\mapsto\b(T,L)$ is a bijection between integral (resp. rational) unbounded measured laminations and $\integers^n$ (resp. $\rationals^n$).
\end{theorem}

\section{Quasi-laminations and the Curve Separation Property}\label{quasi sec}
Theorem~\ref{mut lam} suggests the idea of constructing the mutation fan $\F_B$ and a basis for $B$ in terms of laminations.
It turns out that unbounded measured laminations are almost, but not quite, the right tool.
In this section, we develop the right tool, which we call quasi-laminations.
We also define the Curve Separation Property, a property of $(\S,\M)$ that allows us to construct the rational part of the mutation fan in the context of quasi-laminations.
In Section~\ref{curve separation sec}, we establish the Curve Separation Property for all marked surfaces except those with no boundary components and exactly one puncture.

\begin{definition}[\emph{Quasi-laminations}]\label{quasi lam def}
A curve in $\S$ is \newword{allowable} if it satisfies the requirements placed on curves in Definition~\ref{lam def} and is not one of the following:
\begin{itemize}
\item a curve that has both endpoints on the same boundary segment and that, together with the portion of the boundary between its endpoints, cuts out a once-punctured disk, and 
\item a curve with two coinciding spiral points that cuts out a once-punctured disk.
\end{itemize}
Figure~\ref{exclude curves} illustrates these two excluded curves.
\begin{figure}[ht]
\scalebox{1.2}{\includegraphics{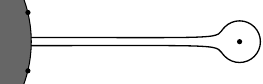}}\qquad\scalebox{1.2}{\includegraphics{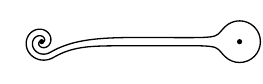}}
\caption{Excluded curves for quasi-laminations}
\label{exclude curves}
\end{figure}

Two allowable curves are \newword{compatible} if they are nonintersecting or if they are identical except that, at one (and only one) end of the curve, they spiral opposite directions into the same point.
Figure~\ref{allow intersect} illustrates the allowed intersections between compatible curves.
One curve is drawn black and the other is dotted and green.
\begin{figure}[ht]
\scalebox{1.2}{\includegraphics{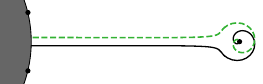}}\qquad\scalebox{1.2}{\includegraphics{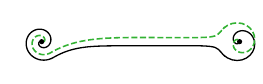}}
\caption{Allowed intersections for quasi-laminations}
\label{allow intersect}
\end{figure}

An \newword{integral quasi-lamination} is a collection of pairwise compatible allowable curves with positive integer weights.
A \newword{rational quasi-lamination} is a collection of pairwise compatible allowable curves with positive rational weights.
\end{definition}

\begin{definition}[\emph{Shear coordinates of quasi-laminations}]\label{quasi-shear def}
Let $T$ be a tagged triangulation, let $\gamma$ be a tagged arc in $T$ and let $L$ be a quasi-lamination.
Then the shear coordinates $\b(T,L)$ are defined as a weighted sum of shear coordinates $\b(T,\lambda)$ of the individual curves $\lambda$ in $L$, just as in Definition~\ref{shear def}.
\end{definition}

We should think of the definition of quasi-laminations as a harmless modification of the definition of laminations.
Indeed, given a rational unbounded measured lamination, one can replace all curves of the type illustrated in Figure~\ref{exclude curves} with pairs of intersecting but compatible curves of the corresponding type shown in Figure~\ref{allow intersect}.
(If some curve $\lambda$ appears twice in the result, we replace the two copies of $\lambda$ with a single copy of $\lambda$ whose weight is the sum of the original two weights.)
This is a bijection, preserving integrality and shear coordinates, between rational unbounded measured laminations and quasi-laminations.
The following theorems are now immediate from Theorems~\ref{mut lam} and~\ref{lam bij}.

\begin{theorem}\label{mut q-lam}
Fix a quasi-lamination $L$.
If $T$ and $T'$ are tagged triangulations such that $T'$ is obtained from $T$ by a flip at an arc $\gamma$, then $\b(T',L)=\eta_\gamma^{B(T)}(\b(T,L))$, where $\eta_\gamma^{B(T)}$ is the mutation map at the index for the arc $\gamma$ in $T$.
\end{theorem}

\begin{theorem}\label{q-lam bij}
Fix a tagged triangulation $T$.
Then the map $L\mapsto\b(T,L)$ is a bijection between integral (resp. rational) quasi-laminations and $\integers^n$ (resp.~$\rationals^n$).
\end{theorem}

In order for us to understand the mutation fan in terms of quasi-laminations, the surface must have the following key property.

\begin{definition}[\emph{Curve Separation Property}]\label{curve sep def}
A marked surface $(\S,\M)$ has the \newword{Curve Separation Property} if for any incompatible allowable curves $\lambda$ and~$\nu$, there exists a tagged triangulation $T$ (with all arcs tagged plain if $(\S,\M)$ has no boundary components and exactly one puncture) and an arc $\gamma$ in $T$ such that the integers $b_\gamma(T,\lambda)$ and $b_\gamma(T,\nu)$ have strictly opposite signs.
\end{definition}

Definition~\ref{curve sep def} defines the Curve Separation Property entirely in the context of the geometry of $(\S,\M)$.
We now rephrase the property in terms of the mutation fan:
Curves are incompatible exactly when they are separated in the mutation fan.

\begin{prop}\label{FBT compatible}
Suppose $(\S,\M)$ satisfies the Curve Separation Property, let $T$ be any tagged triangulation of $(\S,\M)$ (with all arcs tagged plain if $(\S,\M)$ has no boundary components and exactly one puncture), and let $\Lambda$ be a set of allowable curves.
Then the curves in $\Lambda$ are pairwise compatible if and only if the vectors $\set{\b(T,\lambda):\lambda\in\Lambda}$ are contained in some $B(T)$-cone.
\end{prop}
\begin{proof}
Suppose $\lambda$ and $\nu$ are not compatible.
Then the Curve Separation Property says that there exists a tagged triangulation $T'$ (with all arcs tagged plain if $(\S,\M)$ has no boundary components and exactly one puncture) and an arc $\gamma$ in $T'$ such that the integers $b_\gamma(T',\lambda)$ and $b_\gamma(T',\nu)$ have strictly opposite signs.
Then $T'$ is obtained from $T$ by a sequence of flips, so Theorem~\ref{mut q-lam} implies that $\b(T',\lambda)$ and $\b(T',\nu)$ are obtained by applying some mutation map to $\b(T,\lambda)$ and $\b(T,\nu)$.
By Proposition~\ref{contained Bcone}, $\b(T,\lambda)$ and $\b(T,\nu)$ are not contained in any common $B(T)$-cone.

Conversely, let $\lambda$ and $\nu$ be compatible, allowable curves and let $\gamma$ be a tagged arc in a tagged triangulation $T'$.
We will show that $b_\gamma(T',\lambda)$ and $b_\gamma(T',\nu)$ weakly agree in sign.

As in Definition~\ref{shear def}, removing tags and switching spirals on $\lambda$ and $\nu$, we pass from $T'$ to a triangulation $T^\circ$ where $\gamma$ is contained in two distinct triangles.
Let $Q$ be the interior of the quadrilateral that is the union of these two triangles.
We refer to Figure~\ref{shear fig}, where $\gamma$ is represented by a purple arc and where $Q$ is shaded light purple.
(Possibly some of the edges or vertices of $Q$ are identified with each other.)
We see that $b_\gamma(T,\lambda)$ and $b_\gamma(T,\nu)$ do not have opposite signs unless $\lambda$ and $\nu$ intersect, so suppose they intersect.
Since $\lambda$ and $\nu$ are compatible, they are identical except at one end, where they spiral opposite ways into a puncture $p$.
Consider the connected components of $\lambda\cap Q$ and the connected components of $\nu\cap Q$.
Since $\lambda$ and $\nu$ are identical except where they spiral opposite ways into~$p$, each connected component of $\lambda\cap Q$ corresponds to a connected component of $\nu\cap Q$ that intersects $Q$ in the same way, with \emph{one possible exception}:
If $p$ is a vertex of the quadrilateral, there may be one connected component $\lambda'$ of $\lambda\cap Q$ that disagrees with the corresponding component $\nu'$ of $\nu\cap Q$.
But this disagreement can only be slight:
Orienting $\lambda$ towards $p$, the component $\lambda'$ enters $Q$ through some edge of the quadrilateral and leaves $Q$ through some other edge.
Orienting $\nu$ towards $p$ as well, the component $\nu'$ enters $Q$ through the same edge as $\lambda'$ does.
Thus it is impossible for $\lambda'$ and $\nu'$ to contribute strictly opposite signs to the shear coordinates of $\lambda$ and $\nu$.
If either of them (say $\lambda'$) makes a nonzero contribution $\ep$, then since $\lambda$ is not self-intersecting, all components of $\lambda\cap Q$ make contributions weakly agreeing in sign with $\ep$.
Therefore also all components of $\nu\cap Q$ make contributions weakly agreeing in sign with $\ep$, so $b_\gamma(T',\lambda)$ and $b_\gamma(T',\nu)$ weakly agree in sign.

Now if $\Lambda$ is a set of pairwise compatible curves, then $b_\gamma(T',\lambda)$ and $b_\gamma(T',\nu)$ weakly agree in sign for any pair $\lambda,\nu\in\Lambda$.
Therefore, $\set{b_\gamma(T',\lambda):\lambda\in\Lambda}$ all weakly agree in sign for any $T'$.
By Theorem~\ref{mut q-lam}, applying any mutation map $\eta_\k^{B(T)}$ to $\set{\b(T,\lambda):\lambda\in\Lambda}$ produces a sign-coherent set of vectors.
Proposition~\ref{contained Bcone} says that $\set{\b(T,\lambda):\lambda\in\Lambda}$ is contained in some $B(T)$-cone.
\end{proof}

\begin{remark}\label{all arcs tagged plain}
The requirement that all arcs be tagged plain if $(\S,\M)$ has no boundary components and exactly one puncture in Definition~\ref{curve sep def}, in Proposition~\ref{FBT compatible}, and in what follows arises from the fact that the tagged arc complex has two connected components in this case.
(See Definition~\ref{complex and flips}.)
The requirement is crucial in the proof of one direction of Proposition~\ref{FBT compatible}, where it ensures that the tagged triangulations $T$ and $T'$ are related by a sequence of flips.
This requirement is also the reason that, in Section~\ref{curve separation sec}, we are not able to establish the Curve Separation Property for $(\S,\M)$ having no boundary components and exactly one puncture.
\end{remark}

We now define a rational fan $\F_\rationals(T)$, based on quasi-laminations, which is closely related to the mutation fan $\F_{B(T)}$.
Specifically, $\F_\rationals(T)$ is the rational part of $\F_{B(T)}$, in a sense that we make precise below.

\begin{definition}[\emph{Rational quasi-lamination fan}]\label{rat q-lam fan def}
Fix a tagged triangulation $T$.
Let $\Lambda$ be a set of pairwise compatible allowable curves in the sense of Definition~\ref{quasi lam def}.
Let $C_\Lambda$ be the nonnegative $\reals$-linear span of the integer vectors $\set{\b(T,\lambda):\lambda\in\Lambda}$.
Let $\F_\rationals(T)$ be the collection of all such cones $C_\Lambda$.
We call $\F_\rationals(T)$ the \newword{rational quasi-lamination fan} for $T$ and justify the name below in Theorem~\ref{rat FB surfaces}.
\end{definition}

The following is a specialization of \cite[Definition~6.9]{universal}.  %\margin{This is Def. {R-part def}.}
\begin{definition}[\emph{Rational part of a fan}]\label{rat part def}
Suppose $\F$ and $\F'$ are fans such that:
\begin{enumerate}[(i)]
\item \label{rat cones rat part}
$\F'$ is a rational fan.
\item \label{cone in cone rat part}
Each cone in $\F'$ is contained in a cone of $\F$.
\item \label{max cone rat part}
For each cone $C$ of $\F$, there is a unique largest cone (under containment) among cones of $\F'$ contained in $C$.
This largest cone contains $\rationals^n\cap C$.
\end{enumerate}
Then $\F'$ is called the \newword{rational part} of $\F$.
If the rational part of $\F$ exists, then it is unique.
\end{definition}

\begin{theorem}\label{rat FB surfaces}
Let $T$ be any tagged triangulation of $(\S,\M)$ (with all arcs tagged plain if $(\S,\M)$ has no boundary components and exactly one puncture).
The collection $\F_\rationals(T)$ is a rational, simplicial fan.
The fan $\F_\rationals(T)$ is the rational part of $\F_{B(T)}$ if and only if $(\S,\M)$ has the Curve Separation Property.
\end{theorem}

\begin{proof}
Suppose that some cone $C_\Lambda$ is not simplicial.
Then there is some rational point in $C_\Lambda$ that can be written as a $\rationals$-linear combination of the vectors $(\b(T,\lambda):\lambda\in\Lambda)$ in two distinct ways.
This point is thus obtained as the shear coordinates of two distinct rational quasi-laminations, contradicting Theorem~\ref{q-lam bij}.

Let $C_\Lambda$ and $C_{\Lambda'}$ be cones in $\F_\rationals(T)$.
Then $C_\Lambda\cap C_{\Lambda'}$ contains $C_{\Lambda\cap\Lambda'}$. 
Suppose $\v$ is a rational vector in $C_\Lambda\cap C_{\Lambda'}$.
Then $\v$ is the shear coordinates of some quasi-lamination supported on a set $\tilde\Lambda$ of compatible curves.
By Theorem~\ref{q-lam bij}, we conclude that $\tilde\Lambda$ is contained in both $\Lambda$ and $\Lambda'$.
Thus $C_\Lambda\cap C_{\Lambda'}\cap\rationals^n\subseteq C_{\Lambda\cap\Lambda'}$.
Since $C_\Lambda\cap C_{\Lambda'}$ is a rational cone, $C_\Lambda\cap C_{\Lambda'}\subseteq C_{\Lambda\cap\Lambda'}$, so that $C_\Lambda\cap C_{\Lambda'}= C_{\Lambda\cap\Lambda'}$.
We have shown that $\F_\rationals(T)$ is a simplicial fan, and it is rational by construction.

If $(\S,\M)$ does not have the Curve Separation Property, then Theorem~\ref{mut q-lam} and Proposition~\ref{contained Bcone} imply that there exist incompatible allowable curves $\lambda$ and $\nu$ with $\b(T,\lambda)$ and $\b(T,\nu)$ contained in some $B$-cone $C$.
But $\b(T,\lambda)$ and $\b(T,\nu)$ both span rays of $\F_\rationals(T)$ not contained in a common cone of $\F_\rationals(T)$.
Thus $\F_\rationals(T)$ fails condition~\eqref{max cone rat part} of Definition~\ref{rat part def}.

Now suppose $(\S,\M)$ has the Curve Separation Property.
Condition~\eqref{cone in cone rat part} of Definition~\ref{rat part def} holds by Proposition~\ref{FBT compatible}.
We now establish condition~\eqref{max cone rat part} of Definition~\ref{rat part def}.
Given a $B(T)$-cone $C$, let $\Lambda$ be the set of allowable curves $\lambda$ whose shear coordinates $\b(T,\lambda)$ are contained in $C$.
Proposition~\ref{FBT compatible} implies that $C_\Lambda$ is the unique largest cone among cones of $\F_\rationals(T)$ contained in $C$.

It remains to show that this $C_\Lambda$ contains every rational point in $C$.
We argue by induction on the dimension of $C$.
The assertion is trivial when $C$ is the zero cone.
Suppose $\a$ is in $\rationals^n\cap C$.
Then $\a$ is the shear coordinates of some rational quasi-lamination by Theorem~\ref{q-lam bij}.
The shear coordinates of the individual curves in this quasi-lamination span a cone $D$ in $\F_\rationals(T)$.  
The cone $D$ is contained in some $B(T)$-cone $C'$.
If $C'=C$, then $D$ is contained in the largest cone $C_{\Lambda}$ of $\F_\rationals(T)$ in~$C$.
Otherwise, $\a\in C\cap C'$, which is a $B(T)$-cone by Proposition~\ref{int of B-cones} and whose dimension is lower than the dimension of $C$.
Let $D'$ be the largest cone of $\F_\rationals(T)$ contained in $C\cap C'$.
By induction, $\a\in D'$.
But now $D'$ is contained in the largest cone $C_{\Lambda}$ of $\F_\rationals(T)$ in $C$.
In either case, $\a$ is in $C_\Lambda$.
\end{proof}

More generally than Definition~\ref{rat part def}, \cite[Definition~6.9]{universal} defines the $R$-part of a fan for an underlying ring $R$, and \cite[Proposition~6.10]{universal} % \margin{This is Prop. {R part prop}.}
establishes several facts about the $R$-part of a fan.  
The following corollary is obtained from \cite[Proposition~6.10]{universal} by setting $R=\rationals$ and appealing to Theorem~\ref{rat FB surfaces}.

\begin{cor}\label{R part prop}
Suppose $(\S,\M)$ has the Curve Separation Property.
Let $T$ be any tagged triangulation of $(\S,\M)$ (with all arcs tagged plain if $(\S,\M)$ has no boundary components and exactly one puncture).
Each rational $B(T)$-cone is a cone in $\F_\rationals(T)$.
The full-dimensional $B(T)$-cones are exactly the full-dimensional cones in~$\F_\rationals(T)$.
\end{cor}

The Curve Separation Property has important consequences for finding positive $\integers$- and $\rationals$-bases for~$B(T)$.
\begin{theorem}\label{curves pos basis}
Suppose $(\S,\M)$ has the Curve Separation Property.
Let $R$ be $\integers$ or $\rationals$ and fix a tagged triangulation $T$ (with all arcs tagged plain if $(\S,\M)$ has no boundary components and exactly one puncture).
\begin{enumerate}
\item \label{indep allowable}
If shear coordinates $\b(T,\lambda)$ of allowable curves $\lambda$ form an $R$-independent set for $B(T)$, then they form a positive $R$-basis for~$B(T)$.
\item \label{basis allowable}
If a positive $R$-basis exists for $B(T)$, then the shear coordinates of allowable curves are a positive $R$-basis.
\end{enumerate}
\end{theorem}

To prove Theorem~\ref{curves pos basis}, we use the following specialization of \cite[Corollary~6.12]{universal}.
The first part of Proposition~\ref{FB basis ray} as stated in \cite[Corollary~6.12]{universal} refers to the $\integers$-part (as in \cite[Definition~6.9]{universal}), not the rational part, but these two fans coincide.

\begin{prop}\label{FB basis ray}
If a positive $\integers$-basis exists for $B$, then the unique positive $\integers$-basis for $B$ consists of the smallest nonzero integer vector in each ray of the rational part of $\F_B$.
If a positive $\rationals$-basis exists for $B$, then a collection of vectors is a positive $\rationals$-basis for $B$ if and only if it consists of exactly one nonzero vector in each ray of the rational part of $\F_B$.
\end{prop}

\begin{proof}[Proof of Theorem~\ref{curves pos basis}]
For assertion~\eqref{indep allowable}, we verify the hypothesis of Proposition \ref{positive cone basis}.
Suppose $C$ is a $B(T)$-cone.
By Theorem~\ref{rat FB surfaces}, there is a largest cone $D$ among cones of $\F_\rationals(T)$ contained in $C$ and $D$ contains every point in $R^n\cap C$.
This verifies the hypothesis of Proposition~\ref{positive cone basis} for $R=\rationals$.
The hypothesis for $R=\integers$ follows by the integer case of Theorem~\ref{q-lam bij}.
Assertion~\eqref{basis allowable} is an immediate consequence of Theorem~\ref{rat FB surfaces} and Proposition~\ref{FB basis ray}.
(Theorem~\ref{q-lam bij} implies that the shear coordinate vector of an allowable curve is the shortest integer vector in the ray it spans.)
\end{proof}

\begin{example}\label{affine A2 surface}
Consider an annulus with two marked points on its outer boundary and one marked point on its inner boundary.
The surface admits a triangulation with three arcs, for example as shown in Figure~\ref{a2annulus}.
\begin{figure}[ht]
\scalebox{0.8}{\includegraphics{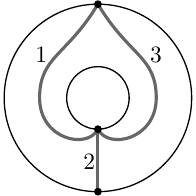}}
\caption{A triangulated annulus}
\label{a2annulus}
\end{figure}
If $T$ is the triangulation shown, then $B(T)$ is the matrix $\begin{bsmallmatrix*}[r]0&1&1\\-1&0&1\\-1&-1&\hspace{6pt}0\\\end{bsmallmatrix*}$.
Some allowable curves in the surface are shown (green and dotted) in Figure~\ref{a2annuluscurves}.
\begin{figure}[ht]
\begin{tabular}{ccccccc}
\scalebox{0.44}{\includegraphics{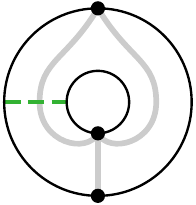}}&
\scalebox{0.44}{\includegraphics{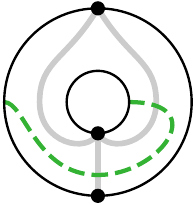}}&
\scalebox{0.44}{\includegraphics{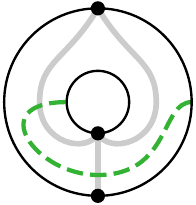}}&
\scalebox{0.44}{\includegraphics{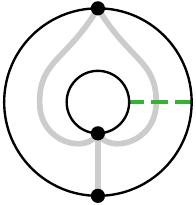}}&
\scalebox{0.44}{\includegraphics{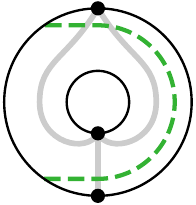}}&
\scalebox{0.44}{\includegraphics{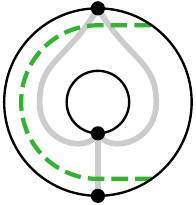}}&
\scalebox{0.44}{\includegraphics{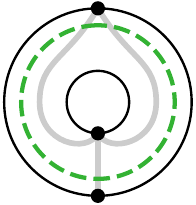}}\\
$\lambda_1$&
$\lambda_2$&
$\lambda_3$&
$\lambda_4$&
$\lambda_+$&
$\lambda_-$&
$\lambda_\infty$\\
$\begin{bsmallmatrix}-1&0&0\end{bsmallmatrix}$&
$\begin{bsmallmatrix}0&1&0\end{bsmallmatrix}$&
$\begin{bsmallmatrix}0&-1&0\end{bsmallmatrix}$&
$\begin{bsmallmatrix}0&0&1\end{bsmallmatrix}$&
$\begin{bsmallmatrix}0&1&-1\end{bsmallmatrix}$&
$\begin{bsmallmatrix}1&-1&0\end{bsmallmatrix}$&
$\begin{bsmallmatrix}1&0&-1\end{bsmallmatrix}$
\end{tabular}
\caption{Some allowable curves in the annulus of Figure~\ref{a2annulus} and their shear coordinates}
\label{a2annuluscurves}
\end{figure}

All remaining allowable curves are related to the curves $\lambda_1$, $\lambda_2$, $\lambda_3$, and $\lambda_4$ as we now explain.
Since each $\lambda\in\set{\lambda_1,\lambda_2,\lambda_3,\lambda_4}$ connects the inner boundary of the annulus to the outer boundary, we can modify $\lambda$ by inserting spirals.
Specifically, for each $n\in\integers$, let $\lambda^{(n)}$ be obtained from $\lambda$ by tracing along $\lambda$ from the inner boundary, inserting $n$ counterclockwise outward spirals, and then continuing along $\lambda$ to the outer boundary.
By convention, $\lambda^{(0)}=\lambda$, and if $n<0$, then $\lambda^{(n)}$ is obtained from $\lambda$ by inserting $-n$ \emph{clockwise} outward spirals.
The curve $\lambda_1^{(2)}$ is illustrated in Figure~\ref{a2annuluscurvespiral}.
\begin{figure}[ht]
\scalebox{.8}{\includegraphics{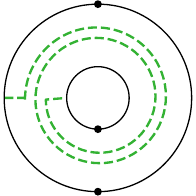}}
\caption{The allowable curve $\lambda_1^{(2)}$}
\label{a2annuluscurvespiral}
\end{figure}
Keeping in mind that allowable curves are considered up to isotopy in $\S$ relative to $\M$, one verifies that $\lambda_1^{(-1)}=\lambda_2$, that $\lambda_2^{(1)}=\lambda_1$, that $\lambda_3^{(-1)}=\lambda_4$, and that $\lambda_4^{(1)}=\lambda_3$.
The set of all allowable curves is 
\[\{\lambda_1^{(n)}:n\ge0\}\cup\{\lambda_2^{(-n)}:n\ge0\}\cup\{\lambda_3^{(n)}:n\ge0\}\cup\{\lambda_4^{(-n)}:n\ge0\}\cup\{\lambda_+,\lambda_-,\lambda_\infty\}.\]

For $i\in\set{1,2,3,4,+,-,\infty}$, let $\v_i\in\integers^3$ be the shear coordinates of $\lambda_i$ with respect to the triangulation shown in Figure~\ref{a2annulus}.
These shear coordinates are also indicated in Figure~\ref{a2annuluscurves}.
%Also, write $\v_i^{(n)}$ for the shear coordinates of 
The shear coordinates of $\lambda_1^{(n)}$ for $n\ge0$ are $\v_1+n\v_\infty$.
The shear coordinates of $\lambda_2^{(-n)}$ for n$\ge0$ are $\v_2+n\v_\infty$, and similar descriptions hold for shear coordinates of $\lambda_3^{(n)}$ and $\lambda_4^{(-n)}$.
The set of all shear coordinates of allowable curves consists of four infinite families $\{\v_i+n\v_\infty:n\ge0\}$ for $i=1,2,3,4$ together with the vectors $\v_+$, $\v_-$, and $\v_\infty$.

We will see in Theorem~\ref{curve separation} that this surface has the Curve Separation Property.
%Alternately, the Curve Separation Property is easily checked directly in this case.
Thus Theorem~\ref{rat FB surfaces} says that $\F_\rationals(T)$ is the rational part of $\F_B$.
In fact, $\F_\rationals(T)$ is complete in this case, so it coincides with $\F_{B(T)}$.
The shear coordinates of allowable curves span the rays of $\F_B$, and the other faces of $\F_B$ are spanned by sets of pairwise compatible allowable curves.

The fan $\F_B$ is depicted in Figure~\ref{FB A fig} in the following way:
\begin{figure}[ht]
\scalebox{1}{\includegraphics{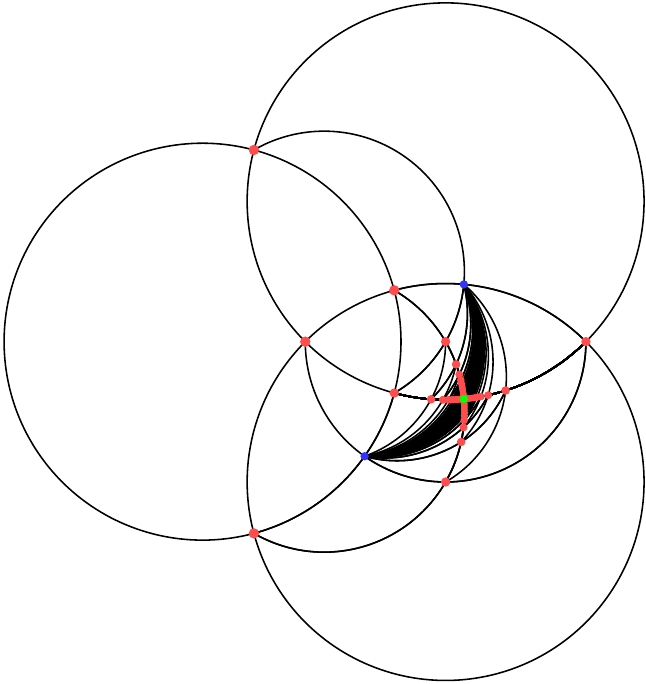}}
\begin{picture}(0,0)(135,-164)
\put(-69,98){$\v_1$}
\put(-2.5,28.5){$\v_2$}
\put(-38,8){$\v_4$}
\put(-67,-102){$\v_3$}
\put(48,31){$\v_+$}
\put(-5.5,-64.5){$\v_-$}
\end{picture}
\caption{$\F_\rationals(T)=\F_{B(T)}$ for $T$ as shown in Figure~\ref{a2annulus}}
\label{FB A fig}
\end{figure}
The intersection of a nonzero cone of $\F_B$ with a unit sphere about the origin is a point, circular arc, or spherical triangle.
The figure shows the stereographic projection of these points, arcs, and triangles.
(The vector $\frac1{\sqrt{3}}[1,\,1,\,1]$ projects to the origin of the plane.)
The ray for $\v_\infty$ is not labeled but is marked as a green dot that appears as a limit of four sequences of red dots.
With two exceptions, each $2$-dimensional face of $\F_B$ is contained in exactly two $3$-dimensional faces of $\F_B$.
The exceptions are the face spanned by $\v_+$ and $\v_\infty$ and the face spanned by $\v_-$ and $\v_\infty$.
\end{example}

\section{Quasi-laminations and \texorpdfstring{$\g$}{g}-vectors}\label{g-vec sec}
In this section, as an aside, we further justify our interest in quasi-laminations by using results of \cite{QP2,NZ} to connect shear coordinates of allowable curves and $\g$-vectors of cluster variables.
Given a tagged arc $\alpha$ in $(\S,\M)$, we define (up to isotopy) a curve $\kappa(\alpha)$.
The curve $\kappa(\alpha)$ coincides with $\alpha$ except within a small ball around the endpoints of $\alpha$.
Within each ball around an endpoint $p$, the behavior of $\kappa(\alpha)$ is as follows:
If $p$ is on a boundary component, then $\kappa(\alpha)$ ends at a point $x$ on the same boundary component such that the path along the boundary from $p$ to $x$, keeping $\S$ on the left, does not leave the small ball.
If $p$ is a puncture, then $\kappa(\alpha)$ spirals into $p$, clockwise if $\alpha$ is tagged plain at $p$ and counterclockwise if $\alpha$ is tagged notched at $p$.
The map $\kappa$ is similar to the map considered in \cite[Definition~17.2]{cats2}, but with the opposite orientation throughout.
The following lemma is immediate.
\begin{lemma}\label{kappa lemma}
The map $\kappa$ from tagged arcs to allowable curves is one-to-one and is surjective onto the set of allowable curves that are not closed.
Two arcs $\alpha$ and $\gamma$ are compatible if and only if the curves $\kappa(\alpha)$ and $\kappa(\gamma)$ are compatible.
\end{lemma}

Let $T$ be a tagged triangulation of $(\S,\M)$.
Part of \cite[Theorem~15.6]{cats2} is a bijection between the tagged arcs in $(\S,\M)$ and the cluster variables of a cluster algebra with initial exchange matrix $B(T)$.
(If $(\S,\M)$ has exactly one puncture and no boundary components, then we assume $T$ has all tags plain and restrict the map to plain-tagged arcs.)
We write $\alpha\mapsto x_{T;\alpha}$ for this bijection, but we do not need the details of the map.
Rather, it is enough to know that tagged triangulations map to clusters.
For a suitable choice of coefficients, each cluster variable $x_{T;\alpha}$ has a $\g$-vector of $\g_{T;\alpha}$.
This is an integer vector with entries indexed by the tagged arcs in $T$, via the bijection between tagged arcs in $T$ and cluster variables in the initial cluster.
For the definition of $\g$-vectors, see \cite[Section~6]{ca4}.

\begin{prop}\label{g surfaces} 
Fix a tagged triangulation $T$ and let $\alpha$ be a tagged arc, not necessarily in $T$.
(If $(\S,\M)$ has exactly one puncture and no boundary components, then take all tags on $T$ and on $\alpha$ to be plain.)
Then $\g_{T;\alpha}$ equals $-\b(T,\kappa(\alpha))$.
\end{prop}

The proof of Proposition~\ref{g surfaces} revolves around the following weak version of \cite[Conjecture~7.12]{ca4}.

\begin{conj}\label{gvec mu}
Let $t_0\dashname{k} t_1$ be an edge in $\T_n$ and let $B_0$ and $B_1$ be exchange matrices such that $B_1=\mu_k(B_0)$.
For any $t\in\T_n$ and $i\in[n]$, the $\g$-vectors $\g_{i;t}^{B_0;t_0}$ and $\g_{i;t}^{B_1;t_1}$ are related by $\g_{i;t}^{B_1;t_1}=\eta^{B_0^T}_k(\g_{i;t}^{B_0;t_0})$, where $B_0^T$ is the transpose of~$B_0$.
\end{conj}

Concatenating results of \cite{QP2,NZ}, as explained in \cite[Remark~8.14]{universal}, %\margin{This is Remark {when sign-coherence known}.}
we see that Conjecture~\ref{gvec mu} is true whenever $B_0$ and $B_1$ are skew-symmetric, and in particular for exchange matrices arising from surfaces.
Furthermore, by skew-symmetry, we rewrite the assertion of Conjecture~\ref{gvec mu} as $\g_{i;t}^{B_1;t_1}=\eta^{-B_0}_k(\g_{i;t}^{B_0;t_0})$, and then use \eqref{eta antipodal} to further rewrite the assertion as
%\begin{equation}\label{surface eta g}
$-\g_{i;t}^{B_1;t_1}=\eta^{B_0}_k(-\g_{i;t}^{B_0;t_0})$,
%\end{equation}
for exchange matrices $B_0$ and $B_1$ related by mutation at position $k$.
Alternately, let $T'$ be obtained from $T$ by a flip of the arc $\gamma$ and let $\alpha$ be any arc. 
Then $-\g(T';\alpha)=\eta_\gamma^{B(T)}(-\g(T,\alpha))$, where $\eta_\gamma^{B(T)}$ is the mutation map at the index for the arc $\gamma$ in $T$.
On the other hand, suppose $\lambda$ is an allowable curve.
Theorem~\ref{mut q-lam} says that the vector $\b(T',\lambda)$ is $\eta_\gamma^{B(T)}(\b(T,\lambda))$, where $\eta_\gamma^{B(T)}$ is the mutation map at the index for the arc $\gamma$ in~$T$.
Setting $\lambda=\kappa(\alpha)$, we see that vectors $-\g(T,\alpha)$ and $\b(T,\kappa(\alpha))$ satisfy the same recurrence.
When $\alpha$ is an arc in $T$, the vector $\b(T,\kappa(\alpha))$ has $-1$ in the position indexed by~$\alpha$ and $0$ elsewhere.
In this case, $x_{T;\alpha}$ is the cluster variable indexed by~$\alpha$ in the initial cluster, so $-\g(T;\alpha)$ also has $-1$ in the position indexed by $\alpha$ and $0$ elsewhere.
This completes the proof of Proposition~\ref{g surfaces}.

\section{The Curve Separation Property in many cases}\label{curve separation sec}
In this section, we prove the following theorem.

\begin{theorem}\label{curve separation}
If  $(\S,\M)$ has one or more boundary components or if $(\S,\M)$ has two or more punctures, then $(\S,\M)$ has the Curve Separation Property.
\end{theorem}

The hypotheses rule out the case where $(\S,\M)$ has no boundary components and exactly one puncture.
See Remarks~\ref{all arcs tagged plain} and~\ref{the missing case}.
As mentioned in the introduction, one additional case of the Curve Separation Property is established in~\cite{unitorus}.

{%These braces are to make the change in section numbering local rather than global.
% Keep this at the beginning of Section {curve separation sec}
\renewcommand{\thesubsection}{\textbf{Case \arabic{subsection}}}
\renewcommand{\thesubsubsection}{Case \arabic{subsection}\alph{subsubsection}}
\newcommand{\case}[1]{\subsection{#1}}
\newcommand{\subcase}[1]{\subsubsection{#1}}

To prove Theorem~\ref{curve separation}, we take incompatible allowable curves $\lambda$ and $\nu$ and construct a tagged triangulation $T$ and a tagged arc $\gamma\in T$ such that the shear coordinates $b_\gamma(T,\lambda)$ and $b_\gamma(T,\nu)$ have strictly opposite signs.
An allowable curve $\lambda$ may intersect $\gamma$ many times, but since $\lambda$ may not intersect itself, all nonzero contributions to $b_\gamma(T,\lambda)$ from these intersections have the same sign.
Thus to prove that $b_\gamma(T,\lambda)$ and $b_\gamma(T,\nu)$ have strictly opposite signs, it is enough to find one intersection of $\lambda$ with $\gamma$ and one intersection of $\nu$ with $\gamma$ that contribute opposite signs to $b_\gamma(T,\lambda)$ and $b_\gamma(T,\nu)$.
We consider many cases.
In most cases we construct, instead of a tagged triangulation $T$, a triangulation $T^\circ$.
In these cases, we tacitly take $T$ to be the tagged triangulation $\tau(T^\circ)$.
(See Definition~\ref{tagged tri def}.)
We avoid restating the conclusion (that the shear coordinates $b_\gamma(T,\lambda)$ and $b_\gamma(T,\nu)$ have strictly opposite signs) in every case.

In each case, part of the triangulation $T$ or $T^\circ$ is illustrated in a figure.
To make the figures as clear as possible, we continue the convention on red, blue, and purple arcs and light-purple shading, established in Definition~\ref{shear def}.
The arc $\gamma$ is always drawn in purple.
Arcs that are not used to calculate shear coordinates at $\gamma$ are gray.
The curve $\lambda$ is solid and black, while $\nu$, when it appears, is dotted and green.
Marked points are numbered as necessary.
We shade areas outside of $\S$ (i.e.\ beyond the boundary) with a dark gray.
Since our figures are confined to the (topologically uncomplicated) screen or page, we often need to indicate that certain loops are not contractible, even though on the page they appear to be contractible.  
To do so, we shade the inside of the loop with a light gray striped pattern.  
The figures show certain points and arcs as distinct when they might possibly coincide.
When necessary, the proof considers the possible coincidences.
In ruling out certain coincidences, it is important to remember that $\S$ is an \emph{oriented} surface.

Some arcs making up $T$ or $T^\circ$ are constructed by \newword{following closely} along existing curves.
This is meant as a technical term:
A curve follows another closely if there are no punctures or boundary components between them.
The point is to be certain that the arcs constructed do indeed bound triangles.
In the figures, one might draw a curve tightly next to another curve to indicate that it follows closely.
However, to aid legibility, in many cases we have straightened curves somewhat.

Because there are many cases, we first list and discuss the cases to make it clear that every possibility is covered.
\begin{description}\addtolength\itemindent{-13pt}
\item[\ref{boundary case}]  Both endpoints of $\lambda$ are on boundary segments.
%\begin{description}
%\item[\ref{two boundary case}] $\lambda$ has endpoints on two distinct boundary segments.
%\item[\ref{one boundary case}] $\lambda$ has both endpoints on the same boundary segment.
%\end{description}
\item[\ref{boundary spiral case}]  $\lambda$ has one endpoint on a boundary segment and one spiral point.
\item[\ref{closed case}]  $\lambda$ is a closed curve.
%\begin{description}
%\item[\ref{closed left right case}] There exists a marked point left of $\lambda$ and a marked point right of $\lambda$.  (This condition is defined in \ref{closed case}.)
%\item[\ref{closed not left right case}] $\lambda$ is a closed curve not falling into \ref{closed left right case}.
%\end{description}
\end{description}
By the symmetry of swapping $\lambda$ and $\nu$, in the remaining cases we are free to assume that both $\lambda$ and $\nu$ have spiral points at both ends.
\begin{description}\addtolength\itemindent{-13pt}
\item[\ref{spiral essential case}]  $\lambda$ and $\nu$ both have spiral points at both ends, and they have an intersection that cannot be removed by changing the directions of spiral points.
\end{description}
In the remaining cases, we can assume that the only intersections between $\lambda$ and $\nu$ come at shared spiral points where their spiral directions disagree.
\begin{description}\addtolength\itemindent{-13pt}
\item[\ref{spiral one not case}]  $\lambda$ and $\nu$ both have spiral points at both ends and only intersect where they share spiral points with opposite directions.  One spiral point of $\nu$ is not a spiral point of $\lambda$.
\end{description}
Again by symmetry, in the remaining case we are free to assume that the set of spiral points of $\lambda$ equals the set of spiral points of $\nu$.
\begin{description}\addtolength\itemindent{-13pt}
\item[\ref{spiral equal case}]  $\lambda$ and $\nu$ both have spiral points at both ends and only intersect where they share spiral points with opposite directions.  The set of spiral points of $\lambda$ equals the set of spiral points of $\nu$.
\end{description}
These cases exhaust the possibilities. 

\begin{remark}\label{the missing case}
When $(\S,\M)$ is a surface with no boundary components and exactly one puncture (the case explicitly ruled out by hypothesis), to prove the Curve Separation Property we must construct a tagged triangulation $T$ and an arc $\gamma$ as described above, with the additional requirement that $T$ has all arcs tagged plain.
For such a surface $(\S,\M)$, only \ref{closed case}, \ref{spiral essential case}, and \ref{spiral equal case} apply.
With a view towards eventually removing the extra hypothesis from the theorem, we give arguments in \ref{closed case} and \ref{spiral equal case} that also work for this kind of surface $(\S,\M)$.
The arguments work because when $\S$ has no boundary components and exactly one puncture, no arc bounds a once-punctured monogon.
Thus when $\S$ has no boundary components and exactly one puncture and we construct $T^\circ$, we know automatically that $T$ has all arcs tagged plain.
We are presently unable to prove \ref{spiral essential case} for $(\S,\M)$ having no boundary components and exactly one puncture.
\end{remark}

We now proceed to consider each case.

\case{Both endpoints of $\lambda$ are on boundary segments.}\label{boundary case}

\subcase{$\lambda$ has endpoints on two distinct boundary segments.}\label{two boundary case}
We color the two boundary segments red and draw blue and purple arcs as shown in Figure~\ref{two boundary fig}.  
\begin{figure}[ht]
\includegraphics{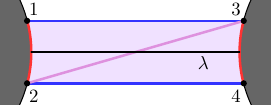}
\caption{An illustration for \ref{two boundary case}}
\label{two boundary fig}
\end{figure}
The top blue arc starts at point $1$, closely follows the left red boundary segment down to $\lambda$, then closely follows $\lambda$ to the right red boundary segment, and closely follows that boundary segment up to point $3$.
We make the symmetric construction for the bottom blue arc.
Since the two red segments are distinct, the blue arcs are compatible or coincide.  
The purple arc is a diagonal of the red-blue-red-blue rectangle and connects points $2$ and $3$.

If the top blue arc bounds an unpunctured monogon, then points $1$ and $3$ coincide, and $\lambda$ is not allowable because it is contractible to a piece of the boundary containing one marked point.
If the top blue arc defines an unpunctured digon, then point~$3$ is a neighbor to point $1$ on the same boundary component.
In this case, we replace the top blue arc by the boundary segment (which we again color blue) whose endpoints are $1$ and $3$.
The symmetric argument shows that the bottom blue arc is valid or can be replaced by a boundary segment.
The purple arc $\gamma$ does not bound an unpunctured monogon because if points $2$ and $3$ coincide, then the two red segments coincide, contradicting the description of this case.
If the purple arc defines an unpunctured digon, then one of the blue arcs bounds an unpunctured monogon, but this has already been ruled out.

Extend these arcs to a triangulation $T^\circ$.
The blue and purple arcs and the red segments define triangles in $T^\circ$.
If $\nu$ is a curve that intersects $\lambda$ (and if the intersection cannot be avoided by independent isotopy deformations of the two curves), then $\nu$ must, somewhere along its extent, enter the pictured quadrilateral through a blue arc (or originate on a blue segment), cross the purple arc, and exit through the opposite blue arc (or end on the opposite blue segment).

\subcase{$\lambda$ has both endpoints on the same boundary segment.}\label{one boundary case}
The situation is illustrated in the top-left picture of Figure~\ref{one boundary fig}.
\begin{figure}[ht]
\scalebox{1.2}{\includegraphics{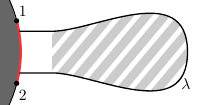}}\qquad\qquad\scalebox{1.2}{\includegraphics{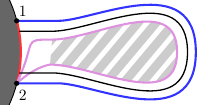}}\\[15 pt]
\scalebox{1.2}{\includegraphics{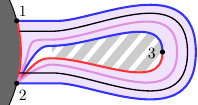}}\qquad\qquad\scalebox{1.2}{\includegraphics{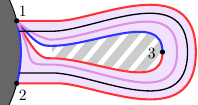}}
\caption{Illustrations for \ref{one boundary case}}
\label{one boundary fig}
\end{figure}
Color the boundary segment red and draw blue and purple arcs as shown in the top-right picture of Figure~\ref{one boundary fig}.
The blue arc follows closely the red segment from point $1$ to $\lambda$, turns left and follows $\lambda$ closely without crossing it, and finally turns left and follows the red segment closely down to point~$2$.
The purple arc $\gamma$ follows the red segment closely up from point $2$, crosses $\lambda$ once and, the second time it meets $\lambda$, turns right and follows $\lambda$ without crossing it until it meets the earlier part of the purple arc, then turns left and follows the red segment closely back down to $2$ without crossing itself.

If the blue arc bounds an unpunctured monogon, then points $1$ and $2$ coincide and $\lambda$ is not allowable because it is contractible to a piece of the boundary containing one marked point.
If the blue arc defines an unpunctured digon, then points $1$ and $2$ are the only marked points on their boundary segment, and we replace the blue arc with the boundary segment connecting $1$ and $2$ that is not already red, coloring this segment blue.
The purple arc does not bound an unpunctured monogon because $\lambda$ is not contractible to a piece of the red boundary segment.
If the purple arc defines an unpunctured digon, then the blue arc bounds an unpunctured monogon, but we have already ruled out that possibility.

Extend these arcs to a triangulation $T^\circ$.
In $T^\circ$, there is a second triangle having the purple arc as an edge.
If the second triangle is self-folded, then $\lambda$ is an excluded curve of a type illustrated in Figure~\ref{exclude curves}.
Thus the second triangle is not self-folded, and we color its other edges as in the bottom-left picture of Figure~\ref{one boundary fig}.

Now suppose $\nu$ intersects $\lambda$.
Then $\nu$ must cross the blue arc (or originate on the blue segment) on the outer side of the picture and then cross either the inner red arc or the inner blue arc.  
If it crosses the inner blue arc, then the shear coordinates $b_\gamma(T,\lambda)$ and $b_\gamma(T,\nu)$ have strictly opposite signs.
Otherwise, we alter the collection of arcs so that the purple arc and the inner red and blue arcs are incident to point $1$ instead of point $2$.
Recoloring appropriately, we obtain the bottom-right picture of Figure~\ref{one boundary fig}.
Now $\nu$ crosses the red arc on the outer side of the picture and then the inner red arc, so the two shear coordinates have strictly opposite signs.

\case{$\lambda$ has one endpoint on a boundary segment and one spiral point.}\label{boundary spiral case}
Suppose first that $\lambda$ spirals clockwise into the puncture.
Color the boundary segment red and label its endpoints $1$ and $2$ as shown in Figure~\ref{boundary puncture fig}.
Label the spiral point as $3$.
\begin{figure}[ht]
\scalebox{1.3}{\includegraphics{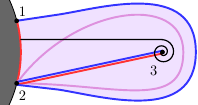}}
\caption{An illustration for \ref{boundary spiral case}}
\label{boundary puncture fig}
\end{figure}
Draw a blue arc from point $1$, following the red segment closely to $\lambda$, following $\lambda$ to point $3$, going closely around point $3$, following $\lambda$ closely back to the red segment, and following the red segment to point~$2$.
The blue arc does not bound an unpunctured monogon because $(\S,\M)$ is not a once-punctured monogon.
If the blue arc defines an unpunctured digon, then points $1$ and $2$ are the only marked points on their boundary component and we replace the blue arc with the non-red boundary segment connecting point $1$ to point $2$, coloring this segment blue.
Draw a purple arc from point $2$ to itself, following the red segment closely from point $2$ to point $1$ and then following the blue arc closely back to point $2$. 
The purple arc does not bound an unpunctured monogon.
It also does not define an unpunctured digon because $(\S,\M)$ is not a once-punctured monogon.
Now draw an arc from point $2$ to the puncture, without crossing the purple arc.
This arc is indicated both in red and in blue in the picture to match the coloring conventions of Figures~\ref{shear fig} and~\ref{shear self-folded fig}.
Complete these arcs to a triangulation $T^\circ$. 

If $\nu$ is not compatible with $\lambda$, then there are two possibilities.
The first possibility is that some portion of $\nu$ crosses the outer blue arc (or originates in the outer blue segment) and then crosses the red/blue arc before crossing the outer blue arc again (or terminating in the outer blue segment). 
In doing so, $\nu$ crosses the red-blue-red-blue quadrilateral twice, picking up a nonzero contribution to $b_\gamma(T,\nu)$ one of the two times.
The second possibility is that $\lambda$ crosses the outer blue arc (or originates in the outer blue segment) and then spirals counterclockwise into the puncture.
In either case, we obtain the desired sign difference of shear coordinates at the purple arc $\gamma$.
Notice that in the case where $\lambda$ and $\nu$ intersect but are compatible (that is, when they agree except for the direction of the spiral into the puncture), the relevant shear coordinate for $\nu$ is zero.

If $\lambda$ spirals counterclockwise into the puncture, then we redraw the picture swapping the roles of points $1$ and $2$, reversing left and right in drawing the arcs, and swapping the colors red and blue.

\case{$\lambda$ is a closed curve.}\label{closed case}
Choose a direction to traverse $\lambda$.
Say a marked point $p$ is \newword{left of $\lambda$} if, traversing $\lambda$ in the chosen direction, there is a path that leaves $\lambda$, goes left, and reaches $p$ without crossing $\lambda$ again.
We define the notion \newword{right of $\lambda$} analogously.
For certain curves $\lambda$, it is possible for the same point to be both right of $\lambda$ and left of $\lambda$.

\subcase{There exists a marked point left of $\lambda$ and a marked point right of $\lambda$.}\label{closed left right case}
Label some marked point left of $\lambda$ with the number $1$.
Label some marked point right of $\lambda$ with the number $2$.
Possibly the points $1$ and $2$ coincide.
We draw a blue arc that follows closely on the left some curve from point $1$ to the left of $\lambda$, turns left, follows $\lambda$ closely on the left all around, turns left again and follows closely the curve back to point $1$.
We draw another blue curve in the same way on the right.  These curves are shown in Figure~\ref{closed annulus fig}.
\begin{figure}[ht]
\includegraphics{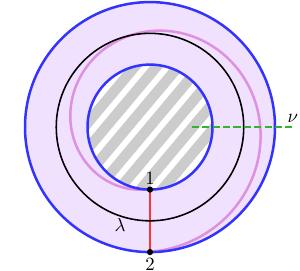}
\caption{An illustration for \ref{closed left right case}}
\label{closed annulus fig}
\end{figure}
Neither blue arc bounds an unpunctured monogon, because $\lambda$ is not contractible to a puncture.
If the blue arc at point $1$ defines an unpunctured digon, then point $1$ is on a boundary component and is the only marked point on that component.
In this case, we replace the inner blue arc with the unique boundary segment on that component, which we then color blue.
(In this case, we should replace the light gray striped shading with dark gray shading.)
We proceed similarly if the blue arc at point $2$ defines an unpunctured digon.
If points $1$ and $2$ coincide, then possibly the two blue curves coincide as well.

If $\nu$ intersects $\lambda$, then somewhere along its extent it crosses the outer blue arc (or originates on the outer blue segment) and then crosses the inner blue arc (or terminates on the inner blue segment).
Possibly $\nu$ winds around the inner blue arc a number of times before crossing it, but for simplicity, it is shown without winding in Figure~\ref{closed annulus fig}.
We draw a red arc from point $1$ to point $2$ without crossing $\nu$ as follows:
The arc follows the inner blue arc closely counterclockwise to $\nu$, then follows $\nu$ closely, without crossing it, to the outer blue arc, and finally follows closely the outer blue arc clockwise to point $2$.
The red and blue arcs define a quadrilateral with the red edges identified.
We draw a purple arc $\gamma$ that forms a diagonal of this quadrilateral.
(Either diagonal works for the conclusion of the theorem, but one choice is consistent with our coloring scheme.)
We complete this set of arcs to a triangulation $T^\circ$.

\subcase{$\lambda$ is a closed curve not falling into \ref{closed left right case}.}\label{closed not left right case}
Without loss of generality, suppose all marked points are left of $\lambda$.
Since these points are not also right of~$\lambda$, the surface decomposes at $\lambda$ as a connected sum.
In Figure~\ref{closed sum fig}, we depict the situation with the ``left'' summand, containing marked points, outside of the curve $\lambda$ and the  ``right'' summand, containing no marked points, inside $\lambda$. 
\begin{figure}[ht]
\scalebox{1.1}{\includegraphics{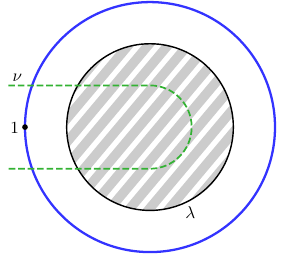}}\qquad\scalebox{1.1}{\includegraphics{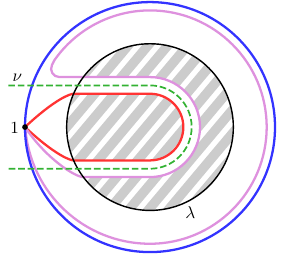}}
\caption{Illustrations for \ref{closed not left right case}}
\label{closed sum fig}
\end{figure}
Since $\lambda$ is not contractible, the right summand is topologically nontrivial.
We choose a particular marked point, labeled $1$, on the left and draw a blue arc closely following $\lambda$ just as we drew the first blue arc in \ref{closed left right case}.
As in \ref{closed left right case}, the blue arc does not bound an unpunctured monogon, and if it defines an unpunctured digon, then we replace it with a blue boundary segment.

If $\nu$ intersects $\lambda$, then $\nu$ must, somewhere along its extent, cross $\lambda$ to enter the right summand, trace a nontrivial loop in the right summand, and then cross $\lambda$ again to leave the right summand.
In fact, $\nu$ may do this several times, but in the description below, we ignore any other pieces of $\nu$ that cross $\lambda$.
The situation is depicted in the left picture of Figure~\ref{closed sum fig}.
In this picture, the entire region inside $\lambda$ is shaded with light gray stripes to indicate both that the right connected summand is topologically nontrivial and that the closed path formed by part of $\nu$ and part of the blue arc is neither contractible to a point nor isotopic to $\lambda$.
Now, draw a red arc and a purple arc as shown in the right picture of Figure~\ref{closed sum fig}.
The red arc follows the blue arc closely from point $1$ up to $\nu$, follows $\nu$ closely without crossing it until reaching the blue arc again, and follows the blue arc back to point $1$.
The purple arc follows the blue arc, crosses $\nu$, and continues almost back to point $1$ before turning to follow $\nu$ closely, without immediately crossing $\nu$ again.
It follows $\nu$ closely until it is close to the blue arc and then crosses $\nu$ to follow the blue arc back to point $1$.
The red and purple arcs do not bound unpunctured monogons or define unpunctured digons because the part of $\nu$ shown is not contractible to part of the blue arc.
The red, blue, and purple arcs form a triangle, oriented consistently with the convention of Figure~\ref{shear fig}.

These arcs can be completed to a triangulation $T^\circ$.
In $T^\circ$, there is a second triangle having the purple arc as an edge, as shown in Figure~\ref{closed sum fig 2}.
\begin{figure}[ht]
\scalebox{1.2}{\includegraphics{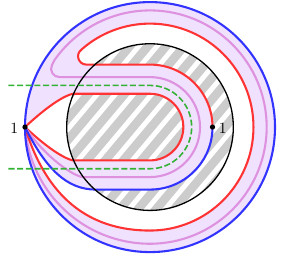}}
\caption{Another illustration for \ref{closed not left right case}}
\label{closed sum fig 2}
\end{figure}
Since point $1$ is the only marked point that can be reached without crossing the blue arc, all three vertices of the second triangle are at point $1$. 
The second triangle is not self-folded, because all three of its vertices coincide.
To compute $b_\gamma(T,\lambda)$, notice that $\lambda$ intersects the red-blue-red-blue quadrilateral in at least two places.
Two intersections are shown in Figure~\ref{closed sum fig 2}.
One of the two contributes $0$ to $b_\gamma(T,\lambda)$, but the other contributes $1$.
A similar situation arises in the calculation of $b_\gamma(T,\nu)$.

An explicit example of \ref{closed not left right case} may be helpful.  
We take $(\S,\M)$ to be the double-torus with one puncture and take $\lambda$ to be a curve along which $\S$ decomposes as a connected sum of tori.
This is a fairly representative example, since the right connected summand of $\S$ in general is a nonempty connected sum of tori.
Specifically, let $\S$ be the octagon with identifications indicated in the top-left picture of Figure~\ref{closed sum example fig}.
\begin{figure}[ht]
\scalebox{1.1}{\includegraphics{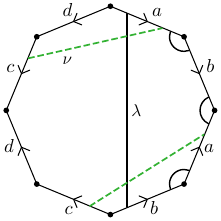}}\qquad\scalebox{1.1}{\includegraphics{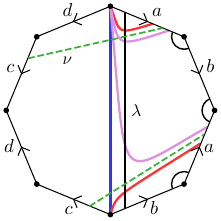}}\\
\scalebox{1.1}{\includegraphics{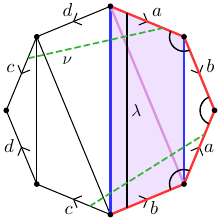}}
\caption{An example of \ref{closed not left right case}}
\label{closed sum example fig}
\end{figure}
All eight vertices of the octagon are identified, and we take this common point to be the puncture.
The curves $\lambda$ and $\nu$ are indicated as usual by a solid black and a dotted green line.
The curve $\lambda$ appears in four pieces that are identified to a closed curve; the curve $\nu$ appears in two pieces.
The top-right picture of Figure~\ref{closed sum example fig} shows the blue, red, and purple arcs constructed first.
The bottom picture shows these three arcs, straightened and completed to a triangulation of $\S$.

\case{$\lambda$ and $\nu$ both have spiral points at both ends, and they have an intersection that cannot be removed by changing the directions of spiral points.}\label{spiral essential case}
Draw a purple arc between the spiral points of $\lambda$, following $\lambda$ closely except for the spirals.
Possibly the endpoints of the purple arc coincide, but the purple arc does not bound an unpunctured or once-punctured monogon because $\lambda$ is not contractible to a spiral point and because $\lambda$ is not an excluded curve in Definition~\ref{quasi lam def}.
Choose a spiral point of $\nu$ and label it $1$.
Label the spiral points of $\lambda$ as $2$ and $3$.
Let $\nu'$ be the part of $\nu$ from the chosen spiral point to the first intersection point of $\lambda$ and $\nu$ that cannot be removed by changing the directions of spiral points.

Draw a red arc that follows $\lambda$ closely from point $2$ without crossing it (except for the spiral) to $\nu'$, then follows $\nu'$ without crossing it (except for the spiral) to point~$1$.
Draw a blue arc from point $3$ to point $1$ closely following $\lambda$ and then $\nu'$.
These arcs are shown in the left picture of Figure~\ref{essential}.
\begin{figure}[ht]
\scalebox{1.1}{\includegraphics{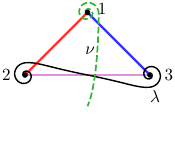}}\qquad\scalebox{1.1}{\includegraphics{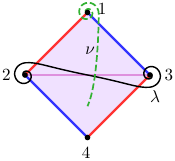}}
\caption{Illustrations for \ref{spiral essential case}}
\label{essential}
\end{figure}
Until further notice, ignore the directions of the spirals in the pictures.
Suppose the red arc bounds an unpunctured monogon.
Then points $1$ and $2$ coincide.
Furthermore, if the spiral directions of $\lambda$ and $\nu$ agree at this point, then this first intersection of $\nu$ with $\lambda$ can be removed by an isotopy deformation of $\nu$.
On the other hand, if spiral directions disagree at this point, then this first intersection can be removed by changing the directions of spiral points.
Either way, we have contradicted our construction of $\nu'$, so the red arc does not bound an unpunctured monogon.
By the same argument, the blue arc also does not bound an unpunctured monogon.
The red and blue arcs are not isotopic because if so, then $\lambda$ is an excluded curve in Definition~\ref{quasi lam def}.

We extend these three arcs to a triangulation.
In this triangulation, there is a second triangle having the purple arc as an edge.
This second triangle is not self-folded, because if so, then $\lambda$ is an excluded curve in Definition~\ref{quasi lam def}.
The situation is illustrated in the right picture of Figure~\ref{essential}, still ignoring directions of spirals.
Possibly point $4$ coincides with some or all of the points $1$, $2$, and $3$.
Let $U$ denote the red-blue-red-blue rectangle shown in the right picture of Figure~\ref{essential}.
After following $\nu'$, the curve $\nu$ exits $U$ either through the bottom blue arc or through the bottom red arc, possibly by spiraling into point~$4$.

Up to a global reorientation, we can assume that $\nu$ either spirals into point $4$ or exits through the bottom blue arc.
Suppose the bottom blue arc is isotopic to the top red arc.
If $\nu$ spirals into point $4$, then $\nu$ is an excluded curve in Definition~\ref{quasi lam def}.
If $\nu$ exits $U$ through the blue arc without spiraling into point $4$, then the situation is as pictured in Figure~\ref{essential isotopic}, still ignoring spiral directions.
\begin{figure}[ht]
\includegraphics{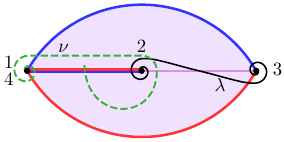}
\caption{Another illustration for \ref{spiral essential case}}
\label{essential isotopic}
\end{figure}
In this case, $\nu$ must spiral into point 2, so it only intersects $\lambda$, if at all, by spiraling the other direction into point $2$.
We conclude from these contradictions that the bottom blue arc is not isotopic to the top red arc.
Possibly, however, the two red arcs are isotopic to each other, the two blue arcs are isotopic to each other, and/or the bottom red arc is isotopic to the top blue arc.

In the triangulation containing these two triangles, there is another triangle sharing the top red arc as an edge, as shown in the left picture of Figure~\ref{essential alter}.
\begin{figure}[ht]
\scalebox{1.1}{\includegraphics{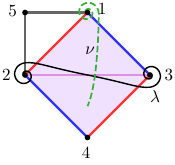}}\qquad\scalebox{1.1}{\includegraphics{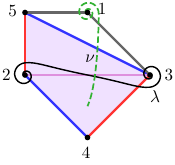}}\qquad\scalebox{1.1}{\includegraphics{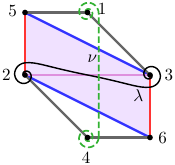}}
\caption{More illustrations for \ref{spiral essential case}}
\label{essential alter}
\end{figure}
Possibly this new triangle is self-folded, but if so, then the red arc is not one of the two edges that are folded onto each other.
We flip the red arc to produce a new red-blue-purple triangle as shown in the middle picture of Figure~\ref{essential alter}.
Let $U'$ be the red-blue-red-blue rectangle consisting of the two triangles having the purple arc as an edge (after the flip of the red edge).
Possibly the two red arcs of the original rectangle $U$ coincide.
In this case, the bottom triangle of $U$ is also affected by the flip of the red arc, and so $\nu$ exits $U'$ through its bottom blue arc, regardless of whether $\nu$ spirals into point $4$ or exits $U$ through the blue arc.  

If $\nu$ exits $U'$ through the blue arc, then we make no more flips and define $U''$ to be $U'$.
Otherwise, the two red arcs of $U$ don't coincide, the bottom triangle of $U'$ coincides with the bottom triangle of $U$, and $\nu$ spirals into point $4$.
In this case, the bottom red arc of $U'$ is not isotopic to the top blue arc of $U'$ for the same reason that the top red arc of $U$ is not isotopic to the bottom blue arc of $U$.
The bottom red and blue arcs of $U'$ are not isotopic (i.e.\ the bottom triangle of $U'$ is not self-folded) for the same reason that the top red and blue arcs of $U$ are not isotopic.
The two red arcs of $U'$ may be isotopic.
We flip the bottom red arc of $U'$, just as we flipped the top red arc of $U$, to obtain a new rectangle $U''$, composed of two triangles sharing the purple arc as an edge, such that $\nu$ enters through the top blue arc and exits through the bottom blue arc.
(The fact that $\nu$ enters $U''$ through the top blue arc follows from the fact that $\nu$ enters $U'$ through the top blue arc, even in the case where the two red arcs of $U'$ are isotopic.)
This construction of $U''$ is illustrated in the right picture of Figure~\ref{essential alter}.

If any red or blue arc defining $U''$ bounds a once-punctured monogon, we replace it by an arc from the vertex of the monogon to the puncture, tagged notched at the puncture.
(We have already ruled out the possibility that the purple arc $\gamma$ bounds a once-punctured monogon.)
The spiral directions of spiral points of $\nu$ are now irrelevant.
If the spiral direction of $\lambda$ at point $2$ disagrees with the direction shown in Figure~\ref{essential alter}, then we tag all arcs notched at point $2$ and similarly for point $3$.
We can do this even if points $2$ and $3$ coincide, because in that case, $\lambda$ has the same spiral directions at both ends.
We complete the tagged arcs to a tagged triangulation $T$, change all tags to plain (reversing spiral directions accordingly) and apply $\tau^{-1}$ to make a triangulation to calculate $b_\gamma(T,\lambda)$ and $b_\gamma(T,\nu)$ as in Definition~\ref{shear def}.
This latter triangulation contains the two triangles of $U''$, and the altered spiral directions are also as shown in the figure.

\case{$\lambda$ and $\nu$ both have spiral points at both ends, and only intersect where they share spiral points with opposite directions.
One spiral point of $\nu$ is not a spiral point of $\lambda$.}\label{spiral one not case}

\subcase{$\lambda$ has two distinct spiral points.}\label{spiral one not distinct case}
Number the spiral points so that $\lambda$ spirals into $1$ and $2$ and $\nu$ spirals into $2$ and $3$. 
Draw a red arc and a blue arc from $1$ to $3$ as shown in the left picture of Figure~\ref{3 spirals fig}.
\begin{figure}[ht]
\includegraphics{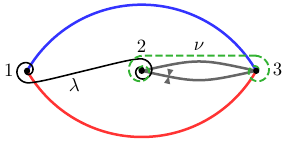}\qquad\includegraphics{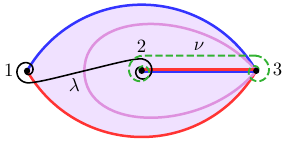}
\caption{Illustrations for \ref{spiral one not distinct case}}
\label{3 spirals fig}
\end{figure}
For the moment, ignore the directions of the spirals in the picture. 
Each arc follows $\lambda$ and then $\nu$ closely.
These two arcs are not isotopic to each other because otherwise $(\S,\M)$ is a sphere with $3$ punctures.
Also draw two gray arcs from $2$ to $3$, following $\nu$ closely and staying between the red arc and the blue arc.
Tag one arc plain at point $2$ and tag the other arc notched at point $2$.
None of these arcs bounds a monogon because points $1$, $2$, and $3$ are distinct.
Tag the arcs at points $1$ and/or $3$ notched if the spiral directions of $\lambda$ or $\nu$ disagree with Figure~\ref{3 spirals fig} at these points.
Complete to a tagged triangulation $T$.
If the spiral directions at point $2$ agree with Figure~\ref{3 spirals fig}, then we take $\gamma$ to be the gray arc tagged notched at point $2$.
If the spiral directions at point $2$ disagree with Figure~\ref{3 spirals fig}, then we take $\gamma$ to be the gray arc tagged plain at point $2$.
In either case, to calculate the shear coordinates $b_\gamma(T,\lambda)$ and $b_\gamma(T,\nu)$, we pass to a triangulation and spiral directions as in the right picture of Figure~\ref{3 spirals fig}.
The purple arc in the right picture is $\gamma$.
The other gray arc in the left picture, besides $\gamma$, becomes the arc in the right picture that we color red on one side and blue on the other side as in Figure~\ref{shear self-folded fig}.

\subcase{$\lambda$ has two coinciding spiral points.}\label{spiral one not coincide case}
The curve $\lambda$ has both spiral points in the same direction at the same point (labeled $1$ in Figure~\ref{curve and loop}).
\begin{figure}[ht]
\scalebox{1.1}{\includegraphics{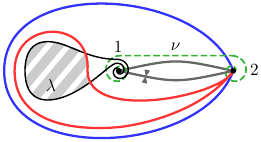}}\qquad\qquad\scalebox{1.1}{\includegraphics{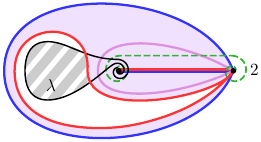}}
\caption{Illustrations for \ref{spiral one not coincide case}}
\label{curve and loop}
\end{figure}
We first draw a red arc as shown in the left picture in Figure~\ref{curve and loop}, starting at point $2$, following $\nu$ closely, crossing $\lambda$ twice near point $1$, following $\lambda$ closely back to the earlier part of the red arc, then following the red arc closely back to point $2$.
This red arc does not bound an unpunctured monogon because if so, then $\lambda$ is contractible to the puncture $1$.
Suppose this red arc bounds a once-punctured monogon.
Then the puncture is either point $1$ or a marked point in the region shaded striped gray in Figure~\ref{curve and loop}.
In either case, $\lambda$ is an excluded curve in Definition~\ref{quasi lam def}, so we conclude that the red arc does not bound a once-punctured monogon.
We next draw a blue arc as shown, following the red arc closely from point $2$ to point $1$ and then following $\nu$ closely back to point $2$.
If the blue arc bounds an unpunctured monogon, then $\lambda$ is an excluded curve in Definition~\ref{quasi lam def}.
If the blue arc bounds a once-punctured monogon, then replace it with an arc from point $2$ to the puncture, tagged notched at the puncture.
Now tag appropriately at points $1$ and $2$ to obtain the correct spiral directions and produce a triangulation, similarly to \ref{spiral one not distinct case}.
Part of the triangulation is shown in the right picture of Figure~\ref{curve and loop}.

\case{$\lambda$ and $\nu$ both have spiral points at both ends, and only intersect where they share spiral points with opposite directions.
The set of spiral points of $\lambda$ equals the set of spiral points of $\nu$.}\label{spiral equal case}

\subcase{The curves have two distinct spiral points.
They coincide, except that the spiral directions disagree in at least one of these spiral points.}\label{spiral equal two coincide case}
In fact, since $\lambda$ and $\nu$ are incompatible, the spiral directions must disagree at both points.
Draw a purple arc, as in Figure~\ref{two coincide fig}, following the two curves closely.
\begin{figure}[ht]
\scalebox{1.1}{\includegraphics{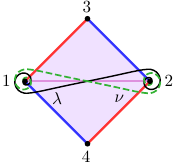}}
\caption{An illustration for \ref{spiral equal two coincide case}}
\label{two coincide fig}
\end{figure}

Since the purple arc has 2 distinct endpoints, it does not bound a monogon.
The purple arc can be completed to some triangulation.
If the purple arc constitutes two sides of a self-folded triangle in the triangulation, then we flip the third side to obtain a new triangulation in which the purple arc is an edge of two distinct triangles.
These two triangles are represented in Figure~\ref{two coincide fig}, although any of the points shown may coincide (except points $1$ and $2$), and various combinations of the arcs shown may coincide.
We apply $\tau$ to the triangulation.
The resulting tagged triangulation has plain tags at points $1$ and $2$ because the purple arc is not part of a self-folded triangle.
We change the plain tagging at point $1$ and/or point $2$ to notched tagging if the spiral directions at these points do not agree with Figure~\ref{two coincide fig}.
The tagged triangulation with these altered taggings is $T$.
When shear coordinates are calculated for the purple arc $\gamma$, we pass to a triangulation and spiral directions as shown in Figure~\ref{two coincide fig}.

\subcase{The curves have two distinct spiral points.
Spiral directions disagree at exactly one of these spiral points.
The two curves do not coincide, even when spiral directions are ignored.}\label{spiral equal two one not coincide case}
Label the two spiral points as $1$ and $2$.
Draw two gray arcs from $1$ to $2$, following $\lambda$ closely.
Tag these two arcs oppositely at point $2$.
Draw a red arc following $\lambda$ closely from point $1$ to point $2$ and then following $\nu$ closely from point $2$ back to point $1$.
Draw a blue arc similarly on the other side of $\lambda$ and $\nu$.
The red and blue arcs are not isotopic, because if so $(\S,\M)$ is a twice-punctured sphere.
Since $\lambda$ and $\nu$ don't coincide, even ignoring spiral directions, neither the red arc nor the blue arc bounds an unpunctured monogon.
If either or both arcs bounds a once-punctured monogon, then replace it with an arc from $1$ to the puncture, tagged notched at the puncture.

\begin{figure}[ht]
\scalebox{1.2}{\includegraphics{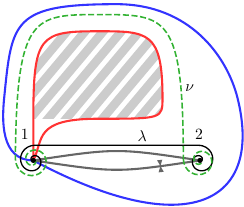}}\qquad\scalebox{1.2}{\includegraphics{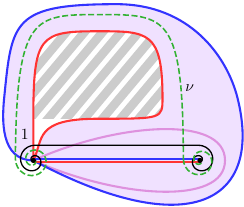}}
\caption{Illustrations for \ref{spiral equal two one not coincide case}}
\label{two spirals disagree once}
\end{figure}

Tag arcs notched at point $1$ if the spiral directions disagree with those shown in Figure~\ref{two spirals disagree once} and complete to  a tagged triangulation $T$.
If the spiral directions at point $2$ agree with Figure~\ref{two spirals disagree once}, then the gray arc tagged notched at point $2$ becomes the purple arc in the right picture of Figure~\ref{3 spirals fig} and the gray arc tagged plain at $2$ becomes the red/blue arc.
If the spiral directions at point $2$ disagree with Figure~\ref{two spirals disagree once}, the two gray arcs reverse roles.
The calculation of shear coordinates at the purple arc involves arcs and spiral directions as shown in the right picture of Figure~\ref{two spirals disagree once}.

\subcase{The curves have two distinct spiral points.
Spiral directions disagree at both of these spiral points.
The two curves do not coincide, even when spiral directions are ignored.}\label{spiral equal two two not coincide case}
Label the two spiral points as $1$ and $2$.
Draw a blue arc, as in the top-left picture in Figure~\ref{two spirals disagree twice}, following $\lambda$ closely.
\begin{figure}[ht]
\scalebox{1.1}{\includegraphics{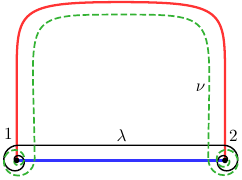}}\qquad\scalebox{1.1}{\includegraphics{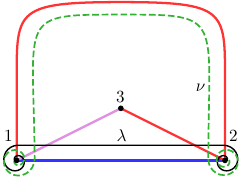}}\\[8pt]
\scalebox{1.1}{\includegraphics{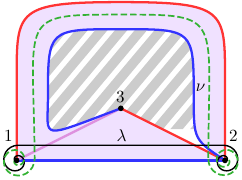}}\qquad\scalebox{1.1}{\includegraphics{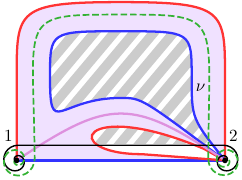}}
\caption{Illustrations for \ref{spiral equal two two not coincide case}}
\label{two spirals disagree twice}
\end{figure}
Draw a red arc, as shown, following $\nu$ closely.
The red and blue arcs can be completed to some triangulation.
This triangulation has two triangles containing the blue arc as an edge.
(The presence of the red edge rules out the possibility that a triangle is folded onto the blue edge.)
One of these triangles is shown in the top-right picture of Figure~\ref{two spirals disagree twice}.
Because points $1$ and $2$ don't coincide, we assume up to symmetry that points $1$ and $3$ don't coincide.
Thus the purple arc does not bound a monogon.
Possibly points $2$ and $3$ coincide, but the new red arc does not bound an unpunctured monogon because it is part of some triangulation.
If the new red arc bounds a once-punctured monogon, then we replace the new red arc with a tagged arc from point $2$ to the puncture.
For the moment, we assume that the purple arc is not isotopic to the original red arc.  
The case where the two are isotopic is dealt with separately below.

Now draw another blue arc from point $3$ to point $2$, closely following the purple arc to point $1$ and then closely following the red arc to point $2$.
This arc is shown in the bottom-left picture of Figure~\ref{two spirals disagree twice} in the case where points $2$ and $3$ do not coincide, and in the bottom-right picture in the case where points $2$ and $3$ coincide.
In the latter case, we need to consider the possibility that the new blue arc bounds an unpunctured or once-punctured monogon.  
But if the new blue arc bounds an unpunctured monogon, then the original red arc is isotopic to the purple arc, a possibility that we deal with separately below.
If the new blue arc bounds a once-punctured monogon, then replace it with an arc from point 2 to the puncture, tagged notched at the puncture. 
Now tag arcs at points $1$ and/or $2$ notched if the spiral directions at these points disagree with the figure, and complete to a tagged triangulation $T$.
To calculate shear coordinates $b_\gamma(T,\lambda)$ and $b_\gamma(T,\nu)$, we pass to a triangulation and spiral directions as shown in one of the bottom pictures of Figure~\ref{two spirals disagree twice}.

Finally, we must consider the case where the original red arc is isotopic to the purple arc.
Figure~\ref{two spirals disagree twice 2} illustrates this case.
\begin{figure}[ht]
\scalebox{1.1}{\includegraphics{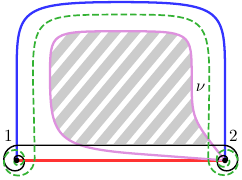}}\qquad\scalebox{1.1}{\includegraphics{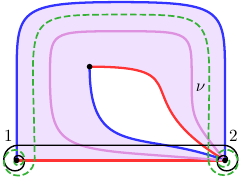}}
\caption{More illustrations for \ref{spiral equal two two not coincide case}}
\label{two spirals disagree twice 2}
\end{figure}
The purple arc is omitted in favor of the original red arc.
Arcs are also colored differently, with the red and blue arcs in the left picture of Figure~\ref{two spirals disagree twice 2} having switched colors from the top-left picture of Figure~\ref{two spirals disagree twice}.
The smaller red arc from the bottom-right picture of Figure~\ref{two spirals disagree twice} is now colored purple. 
We complete these three arcs to a triangulation.
In this triangulation, there is a second triangle having the purple arc as an edge, as illustrated in the right picture of Figure~\ref{two spirals disagree twice 2}.
The new blue and red arcs may be isotopic, in which case the purple arc bounds a once-punctured monogon.  
In this case, we replace the purple arc with an arc from point 2 to the puncture, tagged notched at the puncture.
Again we tag arcs at points $1$ and/or $2$ notched if the spiral directions at these points disagree with the figure, and complete to a tagged triangulation $T$.
When shear coordinates are calculated for the purple arc $\gamma$, we pass to a triangulation and spiral directions as shown in the right picture of Figure~\ref{two spirals disagree twice 2}.
The white space between the red and blue arcs is not shaded striped gray because the red and blue arcs may coincide.

\subcase{The curves have only one spiral point.
They coincide except that spiral directions disagree at their spiral point.}\label{spiral equal one coincide case}
Draw a purple arc following $\lambda$ (and thus~$\nu$) closely.
The purple arc does not bound an unpunctured or once-punctured monogon because $\lambda$ is not contractible to a puncture and is not an excluded curve in Definition~\ref{quasi lam def}.
We complete the purple arc to a triangulation $T^\circ$.
Since the purple arc does not have two distinct endpoints, it does not constitute two edges of a self-folded triangle.
Thus it is contained in two distinct triangles in $T^\circ$.
Furthermore, neither of these triangles is self-folded, because otherwise $\lambda$ is an excluded curve in Definition~\ref{quasi lam def}.
The relevant part of $T^\circ$ is shown in Figure~\ref{one coincide fig}.
\begin{figure}[ht]
\scalebox{1.1}{\includegraphics{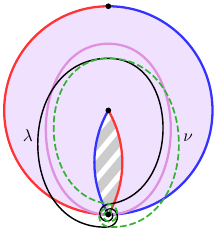}}
\caption{An illustration for \ref{spiral equal one coincide case}}
\label{one coincide fig}
\end{figure}
Up to switching $\lambda$ with $\nu$, the spirals are as shown in the figure.

\subcase{The curves have only one spiral point.
They do not coincide, even when spiral directions are ignored.}\label{spiral equal one not coincide case}
The common spiral point is labeled $1$.
Draw a gray arc following $\lambda$ closely and a purple arc following $\nu$ closely.
\begin{figure}[ht]
\scalebox{1.4}{\includegraphics{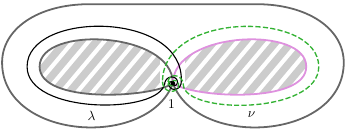}}\\[8pt]
\scalebox{1.4}{\includegraphics{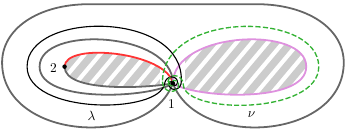}}\\[8pt]
\scalebox{1.4}{\includegraphics{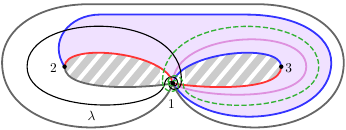}}
\caption{Illustrations for \ref{spiral equal one not coincide case}}
\label{one spirals disagree}
\end{figure}
Then draw a second gray arc following $\lambda$ and then $\nu$ as shown in the top picture of Figure~\ref{one spirals disagree}.
(There are two ways to draw the second gray arc relative to the spiral directions of $\lambda$ and $\nu$.
The right way is shown in the figure.)
Neither the first gray arc nor the purple arc bounds an unpunctured or once-punctured monogon because because neither $\lambda$ nor $\nu$ is contractible to a spiral point or an excluded curve in Definition~\ref{quasi lam def}.
The second gray arc does not bound an unpunctured monogon because $\lambda$ and $\nu$ don't coincide, even when spiral directions are ignored.
For the same reason, the first gray arc and the purple arc are not isotopy equivalent.
The second gray arc is not equivalent to the other two arcs because it is their product in the sense of the fundamental group and neither of the other two arcs is trivial.

These three distinct arcs form a triangle and can be completed to a triangulation.
In this triangulation, there is another triangle having the first gray arc as an edge.
This new triangle is not self-folded because the first gray edge does not bound a once-punctured monogon.
It is shown in the middle picture of Figure~\ref{one spirals disagree}, with one of the new edges red and one gray.
Point $2$ may coincide with point $1$.
We now have two triangles sharing an edge (the first gray edge).
We flip this shared edge, coloring the new edge blue, as shown in the bottom picture of Figure~\ref{one spirals disagree}.
This new configuration of five arcs can also be extended to a new triangulation.
The new triangulation has another triangle sharing the purple arc as an edge.
As before, the new triangle is not self-folded.
It is also shown in the bottom picture of Figure~\ref{one spirals disagree} with the new edges colored red and blue.
Point $3$ may coincide with point $1$ and/or point $2$.
The triangulation constructed has the desired properties with respect to shear coordinates at the purple arc $\gamma$.

}
%These braces are to make the change in section numbering local rather than global.
% Keep this at the end of Section {curve separation sec}

\section{The Null Tangle Property}\label{null tangle sec}
In this section, we formulate the Null Tangle Property of a marked surface, and show that it is equivalent to the assertion that the shear coordinates $\b(T,\lambda)$ of allowable curves $\lambda$ constitute a positive $R$-basis for $B(T)$.
We observe that the Curve Separation Property follows from the Null Tangle Property.
The observation suggests that it might be quite difficult to establish the Null Tangle Property in general, since the proof of the Curve Separation Property (Theorem~\ref{curve separation}) was so complex, even leaving out one family of marked surfaces.
Here, we prove the Null Tangle Property for a smaller family of marked surfaces.
In \cite{unitorus}, the Null Tangle Property is established for an additional marked surface: the once-punctured torus.

%sinceTAMS:Changed this definition to only allow integer weights
\begin{definition}[\emph{Weighted tangle of curves}]\label{weighted tangle def}
A \newword{(weighted) tangle (of curves)} in $(\S,\M)$ is a finite collection of curves, all distinct up to isotopy and all allowable in the sense of Definition~\ref{quasi lam def}, with no requirement of compatibility between curves and with each curve $\lambda$ having a \newword{weight} $w_\lambda\in\integers$.
We do not require weights to be positive as we did in the definitions of laminations and quasi-laminations (Definitions~\ref{lam def} and~\ref{quasi lam def}).
The \newword{support} of a tangle is the set of curves having nonzero weight.
%sinceTAMS: Thus also killed the definition of a rational tangle.
%A \newword{rational tangle} is a weighted tangle of curves with all weights in $\rationals$.
An \newword{straight tangle} is a tangle consisting of pairwise compatible curves.
(The difference between a quasi-lamination and a straight tangle is that the weights in a straight tangle are allowed to be negative.)
A tangle is \newword{trivial} if all weights are zero.

The \newword{weighted union} of tangles $\Xi_1$ and $\Xi_2$ is constructed by first taking the multiset union of $\Xi_1$ and $\Xi_2$.
Then, for each curve $\lambda$ that appears in both $\Xi_1$ and $\Xi_2$, we replace the two copies of $\lambda$ with a single copy of $\lambda$ whose weight is the sum of the original two weights.
\end{definition}

\begin{definition}[\emph{Shear coordinates and the Null Tangle Property}]\label{null tangle def}
A tangle $\Xi$ has shear coordinates $\b(T,\Xi)$ with respect to any tagged triangulation $T$, given by the weighted sum of vectors $\b(T,\lambda)$ over all of the curves $\lambda\in\Xi$.
Each $\b(T,\lambda)$ is calculated as in Definition~\ref{shear def}.
If $(\S,\M)$ has either one or more boundary components or two or more punctures or both, then a \newword{null tangle} in $(\S,\M)$ is a tangle $\Xi$ such that $\b(T,\Xi)$ is the zero vector for all tagged triangulations $T$.
If $(\S,\M)$ has no boundary components and exactly one puncture, then a \newword{null tangle} is a tangle $\Xi$ such that $\b(T,\Xi)$ is the zero vector for all tagged triangulations $T$ with all tags plain.
%sinceTAMS: Thus also killed the word ``rational'' here
A marked surface has the \newword{Null Tangle Property} if every null tangle is trivial.
\end{definition}

The main results of this section are the following theorems and an immediate corollary.

\begin{theorem}\label{null tangle theorem}
Suppose $R$ is $\integers$ or $\rationals$ and let $T$ be a tagged triangulation of $(\S,\M)$ (with all arcs tagged plain if $(\S,\M)$ has no boundary components and exactly one puncture).
The following are equivalent:
\begin{enumerate}[(i)]
\item \label{null tangle theorem null}
$(\S,\M)$ has the Null Tangle Property.
\item \label{null tangle theorem basis}
The shear coordinates of allowable curves form an $R$-basis for $B(T)$.
\item \label{null tangle theorem pos basis}
The shear coordinates of allowable curves form a positive $R$-basis for $B(T)$.
\end{enumerate}
\end{theorem}

\begin{theorem}\label{poly grow null tangle}
If $(\S,\M)$ is a sphere having $b$ boundary components and $p$ punctures with $b+p\le3$, then $(\S,\M)$ has the Null Tangle Property.
\end{theorem}

The hypothesis of Theorem~\ref{poly grow null tangle} is that $(\S,\M)$ is a disk with $0$, $1$, or $2$ punctures, an annulus with $0$ or $1$ punctures, or a sphere with three boundary components.

\begin{cor}\label{poly grow univ}
If $(\S,\M)$ is a sphere having $b$ boundary components and $p$ punctures with $b+p\le3$, then the shear coordinates of allowable curves form a positive $R$-basis for~$B$, for $R=\integers$ or $\rationals$.
\end{cor}

\begin{remark}\label{poly grow remark}
Comparing with \cite[Proposition~11.2]{cats1}, we see that the marked surfaces identified in Theorem~\ref{poly grow null tangle} and Corollary~\ref{poly grow univ} are precisely those whose associated cluster algebra has polynomial growth.
(See \cite[Section~11]{cats1}.)
Given that the Null Tangle Property is established for an additional surface in \cite{unitorus}, the connection to polynomial growth is probably coincidental.
\end{remark}

\begin{remark}\label{finite null tangle remark}
The marked surfaces of Theorem~\ref{poly grow null tangle} and Corollary~\ref{poly grow univ} include those of finite type in the sense of \cite[Section~1.3]{ca2} (the cases of an unpunctured or once-punctured disk).
We recover a weaker version of \cite[Proposition~17.5]{cats2}, giving a proof entirely in terms of triangulated surfaces.
\end{remark}

\begin{remark}\label{affine null tangle remark} 
The marked surfaces of Theorem~\ref{poly grow null tangle} and Corollary~\ref{poly grow univ} also include those of affine type in the sense of \cite[Definition~10.14]{universal} %\margin{This is Definition {def affine type}} 
(the cases of a twice-punctured disk or unpunctured annulus).
For exchange matrices $B$ of affine type, it is conjectured \cite[Conjecture~10.15]{universal} that the $\g$-vectors associated to $B^T$, together with an additional integer vector, are a positive $R$-basis for $B$, for any $R$.
The results of this section establish the marked surfaces case of \cite[Conjecture~10.15]{universal}, %\margin{This is Conjecture {affine univ conj}}
 using Proposition~\ref{g surfaces} and the fact that there is a unique allowable closed curve in these cases.
 Cf. Example~\ref{affine A2 surface}.
\end{remark}

%sinceTAMS: Added following prose and proposition
The point of defining null tangles is the connection to $B$-coherent linear relations.

\begin{prop}\label{null tangle is B coher rel}
A tangle $\Xi$ in $(\S,\M)$ is null if and only if $\sum_{\lambda\in\Xi}w_\lambda\b(T,\lambda)$ is a $B(T)$-coherent linear relation for any tagged triangulation $T$ (with all tags plain if $(\S,\M)$ has no boundary components and exactly one puncture).
\end{prop}
\begin{proof}
By Theorem~\ref{mut q-lam}, a tangle $\Xi$ is null if and only if $\sum_{\lambda\in\Xi}w_\lambda\b(T,\lambda)$ is a nontrivial linear relation among the shear coordinates of allowable curves that is preserved (in the sense of \eqref{linear eta}) under mutation maps at $B(T)$.
Now Proposition~\ref{no zero no piecewise} says that $\Xi$ is null if and only if $\sum_{\lambda\in\Xi}w_\lambda\b(T,\lambda)$ is a $B(T)$-coherent linear relation.
\end{proof}

%sinceTAMS: rephrased next sentence and deleted proof of Prop. {null tangle prop}
%in light of Prop {null tangle is B coher rel}
Proposition~\ref{null tangle is B coher rel} implies the following rephrasing of the Null Tangle Property. 

\begin{prop}\label{null tangle prop}
Suppose $R$ is $\integers$ or $\rationals$ and let $T$ be a tagged triangulation of $(\S,\M)$ (with all arcs tagged plain if $(\S,\M)$ has no boundary components and exactly one puncture).
Then $(\S,\M)$ has the Null Tangle Property if and only if the shear coordinates of allowable curves constitute an $R$-independent set for $B(T)$.  
\end{prop}

%sinceTAMS: added following prose and prop.
We now quote a version of \cite[Proposition~4.12]{universal}. %\margin{This is Prop {one positive or one negative}}
The proposition is stated in \cite{universal} in terms of $B$-coherent linear relations, but we use Proposition~\ref{null tangle is B coher rel} to state a specialization of the proposition in terms of tangles.
\begin{prop}\label{one positive or one negative}
Let $\Xi$ be a tangle in $(\S,\M)$.
Suppose for some tagged triangulation $T$ (with all tags plain if $(\S,\M)$ has no boundary components and exactly one puncture), for some tagged arc $\gamma$ in $T$, and for some curve $\lambda$ in $\Xi$ that $b_\gamma(T,\lambda)$ is strictly positive (resp. strictly negative) and that $b_\gamma(T,\nu)$ is nonpositive (resp. nonnegative) for every other $\nu\in\Xi$.
Then $w_\lambda=0$.
\end{prop}

We now explain the connection between the Null Tangle Property and the Curve Separation Property.
The \newword{disorder} of a tangle $\Xi$ is the smallest number $k$ such that, after deleting curves with weight zero, $\Xi$ is a weighted union of $k$ straight tangles.
Then the Null Tangle Property is the statement that the disorder of a null tangle must be $0$.
One approach to establishing the property would be to rule out null tangles of disorder $k$ for all positive $k$.
The Curve Separation Property is closely related to ruling out null tangles of certain small disorders.

\begin{prop}\label{no 1 2}
Null tangles of disorder $1$ do not exist in any marked surface.
If $(\S,\M)$ has the Curve Separation Property, then null tangles of disorder $2$ do not exist in $(\S,\M)$.
\end{prop}
\begin{proof}
Theorem~\ref{rat FB surfaces}, specifically the fact that $\F_\rationals(T)$ is simplicial, implies that a null tangle cannot have disorder $1$.
Suppose $(\S,\M)$ has the Curve Separation Property.
If a null tangle $\Xi$ has disorder $2$, then it is a weighted union of two nontrivial straight tangles $\Xi_1$ and $\Xi_2$, and there is a curve $\lambda$ in the support of $\Xi_1$ and a curve $\nu$ in the support of $\Xi_2$ such that $\lambda$ and $\nu$ are not compatible.
The Curve Separation Property says that there is a tagged triangulation $T$ and an arc $\gamma$ (with all arcs tagged plain if $(\S,\M)$ has no boundary components and exactly one puncture) such that $b_\gamma(T,\lambda)$ and $b_\gamma(T,\nu)$ have strictly opposite signs.
Without loss of generality, take $b_\gamma(T,\lambda)$ positive and $b_\gamma(T,\nu)$ negative.
Proposition~\ref{FBT compatible} says that $\Xi_1$ is contained in a $B(T)$-cone, so Proposition~\ref{contained Bcone} says that $b_\gamma(T,\lambda')$ is nonnegative for every curve $\lambda'$ in $\Xi_1$.
Similarly $b_\gamma(T,\nu')$ is nonpositive for every curve $\nu'$ in $\Xi_2$.

Since $\Xi$ is a null tangle, we have $\b(T,\Xi_1)=-\b(T,\Xi_2)$.
Let $T'$ be the tagged triangulation obtained by flipping $\gamma$ in $T$.
Since the shear coordinates $b_\gamma(T,\lambda')$ all have weakly the same sign, Theorem~\ref{mut q-lam} (applied separately to each curve in $\Xi_1$) implies that $\b(T',\Xi_1)=\eta_\gamma^{B(T)}(\b(T,\Xi_1))$.
For the same reason, $\b(T',\Xi_2)=\eta_\gamma^{B(T)}(\b(T,\Xi_2))$.
But $B(T)$ does not have a row consisting all of zeros, so it does not have a column consisting all of zeros.
Thus since $\b(T,\Xi_1)$ and $\b(T,\Xi_2)$ have strictly opposite signs in the entry indexed by~$\gamma$, looking at~\eqref{mutation map def} we conclude  that $\b(T',\Xi_1)\neq-\b(T',\Xi_2)$.
This contradicts the assumption that $\Xi$ is a null tangle.
\end{proof}

\begin{prop}\label{null disorder 3 curve sep}
If $(\S,\M)$ does not admit null tangles of disorder $2$ or~$3$, then $(\S,\M)$ has the Curve Separation Property.
\end{prop}
\begin{proof}
Suppose the Curve Separation Property fails.
Then there exist incompatible curves $\lambda$ and $\nu$ such that $b_\gamma(T',\lambda)$ and $b_\gamma(T',\nu)$ have weakly the same sign for all $T'$ and $\gamma\in T'$.
(If $(\S,\M)$ has no boundary components and exactly one puncture, then instead we know that $b_\gamma(T',\lambda)$ and $b_\gamma(T',\nu)$ have weakly the same sign for all $T'$ with all arcs tagged plain.)
Applying Theorem~\ref{mut q-lam}, we see that $\eta_\k^{B(T)}(\b(T,\lambda))$ and $\eta_\k^{B(T)}(\b(T,\nu))$ form a sign-coherent set for any sequence $\k$ of indices.
Proposition~\ref{contained Bcone} says that $\b(T,\lambda)$ and $\b(T,\nu)$ are contained in some $B$-cone, so $\eta_\k^{B(T)}(\b(T,\lambda)+\b(T,\nu))$ equals $\eta_\k^{B(T)}(\b(T,\lambda))+\eta_\k^{B(T)}(\b(T,\nu))$ for any~$\k$ by Proposition~\ref{linear}.
Now let $L$ be the unique quasi-lamination with $\b(T,L)=\b(T,\lambda)+\b(T,\nu)$.
Let $\Xi_1$ be the tangle containing only the curve $\lambda$ with weight $1$ and let $\Xi_2$ be the tangle containing only $\nu$ with weight $1$.
Let $\Xi_3$ be the tangle obtained from $L$ by negating all weights.
The weighted union of $\Xi_1$, $\Xi_2$, and $\Xi_3$ is a null tangle of disorder $d\le3$.
But $d\neq0$ because $\lambda$ and $\nu$ are incompatible.
Proposition~\ref{no 1 2} rules out the possibility $d=1$, so $d$ is $2$ or~$3$.
\end{proof}

\begin{cor}\label{null implies sep}
If $(\S,\M)$ has the Null Tangle Property, then it also has the Curve Separation Property.
\end{cor}

Combining results, we obtain Theorem~\ref{null tangle theorem}.

\begin{proof}[Proof of Theorem~\ref{null tangle theorem}]
Given \eqref{null tangle theorem null}, Theorem~\ref{curves pos basis}, Corollary~\ref{null implies sep} and Proposition \ref{null tangle prop} prove \eqref{null tangle theorem pos basis}, which trivially implies \eqref{null tangle theorem basis}, which implies~\eqref{null tangle theorem null} by Proposition~\ref{null tangle prop}.
\end{proof}

The relationship between the Curve Separation Property and the Null Tangle Property suggests that we might extend the proof of Theorem~\ref{curve separation} to establish the Null Tangle Property under the hypotheses of Theorem~\ref{curve separation}.
While this seems hard in general, some of the easier cases of the proof of the theorem extend to partial statements about the Null Tangle Property, which lead to the proof of Theorem~\ref{poly grow null tangle}.

\begin{prop}\label{null tangle two boundary}
If $\Xi$ is a null tangle, then no curve in the support of $\Xi$ connects two distinct boundary segments.
\end{prop}

\begin{proof}
Suppose $\lambda$ is an allowable curve connecting two distinct boundary segments.
Construct a tagged triangulation $T'$ and distinguish a (purple) arc $\gamma\in T'$ as in \ref{two boundary case} of the proof of Theorem~\ref{curve separation}.
(See Figure~\ref{two boundary fig}.)
We observe that $\lambda$ is the unique allowable curve whose shear coordinate $b_\gamma(T,\lambda)$ is strictly positive and apply Proposition~\ref{one positive or one negative} to conclude that $\lambda$ cannot appear in the support of a null tangle.
\end{proof}

\begin{prop}\label{null tangle boundary puncture}
If $\Xi$ is a null tangle, then no curve in the support of $\Xi$ has one endpoint on a boundary segment and spirals into a marked point at the other end.
\end{prop}
\begin{proof}
Suppose $\lambda$ is an allowable curve having one endpoint on a boundary segment and a spiral point at the other end.
Construct $T'$ and the purple arc $\gamma\in T'$ as in \ref{boundary spiral case} of the proof of Theorem~\ref{curve separation}.
(See Figure~\ref{boundary puncture fig}.)
Again, we observe that $\lambda$ is the unique allowable curve whose shear coordinate $b_\gamma(T,\lambda)$ is strictly positive and apply Proposition~\ref{one positive or one negative}.
\end{proof}

\begin{prop}\label{null tangle one boundary}
Suppose $(\S,\M)$ is a sphere having $b$ boundary components and $p$ punctures, with $b+p\le3$.
If $\Xi$ is a null tangle, then no curve in the support of $\Xi$ has both endpoints on the same boundary segment.
\end{prop}
\begin{proof}
Suppose $\Xi$ is a null tangle and let $\lambda$ be a curve in $\Xi$ with both endpoints on the same boundary segment.
In particular, $(\S,\M)$ is not a disk or once-punctured disk.
Construct $T'$ and the purple arc $\gamma\in T'$ in the bottom-left picture of Figure~\ref{one boundary fig}, in \ref{one boundary case} of the proof of Theorem~\ref{curve separation}.
Suppose $\nu$ is a curve in $\Xi$ with $b_\gamma(T,\nu)>0$.
Then $\nu$ has an endpoint on the same boundary segment as $\lambda$, so Propositions~\ref{null tangle two boundary} and~\ref{null tangle boundary puncture} imply that $\nu$ has both endpoints on the same boundary segment as $\lambda$.

If $(\S,\M)$ is a disk with $2$ punctures or an unpunctured annulus, then each boundary segment admits a unique allowable curve with both endpoints on that boundary segment.
Thus $\lambda$ is the unique curve in $\Xi$ with $b_\gamma(T,\lambda)>0$.
We apply Proposition~\ref{one positive or one negative} to see that $\lambda$ is not in the support of $\Xi$.

If $(\S,\M)$ is a sphere with three boundary components or a once-punctured annulus, then we argue in several steps, looking at the bottom-left picture of Figure~\ref{one boundary fig}.
If point $3$ is the puncture and the other boundary component is not contained in the gray striped area, then $\lambda$ is not allowable in Definition~\ref{quasi lam def}.
If point $3$ is on a boundary component and no marked point off of that boundary component is contained in the gray striped area, then any curve $\nu$ in $\Xi$ with $b_\gamma(T,\nu)$ positive is isotopic to $\lambda$.
Thus again $\lambda$ is the unique curve in $\Xi$ whose shear coordinate $b_\gamma(T,\lambda)$ is strictly positive, so again $\lambda$ is not in the support of $\Xi$.
We are left with a unique choice of $\lambda$ that might appear in the support of $\Xi$, namely the curve such that the area shaded striped gray contains two boundary components or a boundary component and a puncture.
But then $\lambda$ would be the unique curve in $\Xi$ with shear coordinate $b_\gamma(T,\lambda)$ strictly positive, so it is not in the support of~$\Xi$.
\end{proof}

\begin{proof}[Proof of Theorem~\ref{poly grow null tangle}]
Theorem~\ref{curve separation} says that the Curve Separation Property holds, so by Proposition~\ref{no 1 2}, null tangles of disorder $1$ and $2$ do not exist.

If $(\S,\M)$ is a disk with $0$ or $1$ punctures, then Propositions~\ref{null tangle two boundary} and~\ref{null tangle boundary puncture} amount to the Null Tangle Property.
In the remaining cases, Propositions~\ref{null tangle two boundary}, \ref{null tangle boundary puncture}, and~\ref{null tangle one boundary} show that the support of a null tangle has no curves that involve boundary segments.

If $(\S,\M)$ is an unpunctured annulus, there is one remaining allowable curve (the closed curve). 
A null tangle has disorder at most $1$, and therefore has disorder~$0$.

If $(\S,\M)$ is a disk with $2$ punctures, then there are five remaining allowable curves: one closed curve and four curves with distinct spiral points.
These can be arranged into two sets of pairwise compatible curves.
(Group together the two curves with counterclockwise spirals at one of the punctures, group together the two curves with clockwise spirals at that same puncture, and put the closed curve in either group.)
A null tangle has disorder at most $2$ and therefore has disorder~$0$.

If $(\S,\M)$ is a once-punctured annulus, then there are four remaining allowable curves: two closed curves and two curves which spiral into the puncture at both ends.
A null tangle has disorder at most $2$ and thus has disorder $0$.

If $(\S,\M)$ is a sphere with three boundary components, then the only remaining allowable curves are three compatible closed curves, so a null tangle has disorder at most $1$, and therefore has disorder $0$.
\end{proof}

\addtocontents{toc}{\mbox{ }}

\section*{Acknowledgments}
The author thanks Sergey Fomin for helpful converations.
The author also thanks an anonymous referee for many helpful comments, and especially for pointing out that an earlier version of  \cite[Proposition~4.6]{universal} (appearing here as Proposition~\ref{basis exists}) was incorrect.

\end{document}